\documentclass[12pt]{amsart}
\usepackage[margin=1in]{geometry}
\usepackage{graphicx}

\usepackage{mathrsfs}
\usepackage{subfigure}
\usepackage{amsmath}
\usepackage{float}
\usepackage{amssymb}
\usepackage{amsthm}
\usepackage{tikz}
\usetikzlibrary{arrows,positioning}
\usetikzlibrary{fit,matrix,calc,backgrounds}
\usepackage[utf8]{inputenc}
\usepackage{ytableau}
\usepackage{mathdots}
\usepackage{mathtools}
\usepackage{enumitem}
\usepackage{verbatim}

\DeclareMathOperator{\set}{set}

\usepackage{hyperref}
 \theoremstyle{definition}
 \newtheorem{theorem}{Theorem}[section]
 \newtheorem{example}[theorem]{Example}
 \newtheorem{corollary}[theorem]{Corollary}
 \newtheorem{definition}[theorem]{Definition}
 \newtheorem{lemma}[theorem]{Lemma}
 \newtheorem*{remark*}{Remark}
\newtheorem*{note*}{Note}
\newtheorem*{notation}{Notation}
\usepackage{bbm}
\newcommand{\nsym}{\textbf{NSym}}
\newcommand{\qsym}{\textnormal{QSym}}
\newcommand{\sym}{\Lambda}
\newcommand{\xx}{\boldsymbol{x}}
\newcommand{\XX}{\boldsymbol{X}}
\newcommand{\nr}{\boldsymbol{r}}
\newcommand{\nel}{\boldsymbol{e}}
\newcommand{\nh}{\boldsymbol{h}}
\newcommand{\np}{\boldsymbol{p}}
\newcommand{\nphi}{\boldsymbol{\phi}}
\newcommand{\npsi}{\boldsymbol{\psi}}

\newcommand{\modelfix}{\models_w}
\newcommand{\modelsstrong}{\models}
\DeclareMathOperator{\For}{\textnormal{For}}
\DeclareMathOperator{\des}{des}
\DeclareMathOperator{\pb}{pb}
\DeclareMathOperator{\fb}{fb}

\title{A Combinatorial Perspective on the Noncommutative Symmetric Functions}
\author{Angela Hicks and Robert McCloskey}

\begin{document}

\begin{abstract} The noncommutative symmetric functions $\textbf{NSym}$ were first defined abstractly by Gelfand et al.\ in 1995 as the free associative algebra generated by noncommuting indeterminates $\{\boldsymbol{e}_n\}_{n\in \mathbb{N}}$ that were taken as a noncommutative analogue of the elementary symmetric functions.  The resulting space was thus a variation on the traditional symmetric functions $\sym$.  Giving noncommutative analogues of generating function relations for other bases of $\sym$ allowed Gelfand et al.\ to define additional bases of $\textbf{NSym}$ and then determine change-of-basis formulas using quasideterminants.  In this paper, we aim for a self-contained exposition that expresses these bases concretely as functions in infinitely many noncommuting variables and avoids quasideterminants.  Additionally, we look at the noncommutative analogues of two different interpretations of change of basis in $\sym$: both as a product of a minimal number of matrices, mimicking Macdonald's exposition of $\sym$ in \textit{Symmetric Functions and Hall Polynomials}, and as statistics on brick tabloids, as in work by E\u{g}ecio\u{g}lu and Remmel, 1990.
\end{abstract}
\maketitle
\section{Introduction}
\noindent We define three well-known vector spaces: the symmetric functions, $\sym$, the quasisymmetric functions, $\qsym$, and lastly, the noncommutative symmetric functions, $\nsym$, the focus of this work.  The noncommutative symmetric functions were first defined in \cite{non} by Gelfand et al., with a definition inspired by the well-known Fundamental Theorem of Symmetric Polynomials.  Let $\mathbb{N}$ be the set of all nonnegative integers.
\begin{theorem}  Every symmetric function can be written uniquely as a polynomial in the elementary symmetric functions $\{e_n\}_{n\in \mathbb{N}}$.
\end{theorem}

\noindent Inherent to the statement is the fact that the elementary symmetric functions commute with each other.  The noncommutative symmetric functions were originally defined by Gelfand et al.\ in \cite{non}, and answer the following question.
\begin{quote}
    ``What happens if the elementary symmetric functions are replaced by a noncommutative multiplicative basis?''
\end{quote}
\noindent Formally,  Gelfand et al.\ define:
\begin{definition}\label{def:Nsymbyformal}[\cite{non}, section 3] The ring of \textbf{noncommutative symmetric functions} is the free associative algebra 
$$\textbf{NSym} = \langle \boldsymbol{e}_n\rangle_{n \in \mathbb{Z}^+},$$
where $\{\boldsymbol{e}_n\}_{n \in \mathbb{Z}^+}$ is a set of algebraically independent indeterminates which do not commute.  By convention, we occasionally also use $\boldsymbol{e}_0$ for 1, to parallel the symmetric function literature, although one should be careful, since $\boldsymbol{e}_0$ commutes with the remainder.  The symbol $\boldsymbol{e}_n$ is the $n^{\text{th}}$ \textbf{noncommutative elementary symmetric function}.  
\end{definition}

\noindent Gelfand et al.\ proceed to define additional bases by generalizing the relations between generating series of various well-known bases of $\sym$.  They then determine explicit change-of-basis formulas between the resulting bases, using their previous work on quasideterminants as a key tool.  Finally, they mention briefly a concrete realization of several of their newly defined bases
in terms of a set of noncommuting variables on pages 73-74 of their work, but do not give an explicit characterization of the resulting space.  Thus their presentation, while inherently logical to its inspiration, is the reverse of the traditional presentation of the symmetric functions in well-known texts in this area (e.g., Stanley, Chapter 7 in \cite{ECII}, or Macdonald \cite{Macdonald}).  Such texts start with concrete realizations of these spaces as subspaces of the space of power series in infinitely many commuting variables, define explicit bases in those variables, and then derive relations from those definitions.  This convention continues in texts which cover the noncommutative symmetric functions.  For example, Luoto et al.\ in \cite{intro} give a friendly, well-written introduction to $\sym$, $\qsym$, and $\nsym$.  They define the first two as subspaces of the space of formal power series in infinitely many variables, but define $\nsym$ using Definition \ref{def:Nsymbyformal} above.

A few other sources also mention the same concrete realization of bases of $\nsym$  as formal power series in noncommuting variables mentioned in Gelfand et al. Huang \cite{Huang} is perhaps the closest to this work; on page 16 $\nsym$ is defined as the span of the noncommutative complete homogeneous functions, given in terms of noncommuting variables.  The noncommutative ribbon and elementary symmetric functions are also given in terms of the same set of variables, but the work does not prove this definition gives an isomorphic space to that of Gelfand et al.  The proof of the equivalence is given indirectly in \cite{grinberg2014hopf} in Corollary 8.1.14 and is implicit in the proof of Proposition 6.2 in Meliot's text \cite{meliot2017representation}.  To the best of our knowledge, Definition \ref{def:Nsymbyalph} and Corollary \ref{cor:free} below, which give an explicit characterization of when a polynomial in $n$ noncommuting variables is in $\nsym$ (and show that this characterization gives an isomorphic space to that of Gelfand et al.) do not clearly appear elsewhere in the literature, despite their simplicity.

Our goal in this work is an elementary (pun intended) exposition of the noncommutative symmetric functions, following traditional presentations of the symmetric functions by emphasizing their realization in terms of a noncommuting set of variables, then deriving the defining relations of Gelfand et al., emphasizing their similarity to the well-studied symmetric functions, all while avoiding quasideterminants.  Along the way, we will show that many of the change-of-basis matrices between well-known bases of $\nsym$ that were derived in Gelfand et al.\ can also be interpreted to generalize well-studied change-of-basis results in $\sym$.  Particularly, we will give a natural generalization of brick tabloids, introduced by E\u{g}ecio\u{g}lu and Remmel in \cite{brick} to unify combinatorial interpretations for a number of basis transitions in $\sym$, and consider a generalization of the commuting diagram Macdonald gives in \cite{Macdonald} to express change of basis as a product of a few key transition matrices.

\subsection{Notation}
Before we begin, we offer a few brief remarks on notation.  In order to distinguish between commuting and noncommuting variables, let $X= (x_1,x_2,\dots)$ give an infinite sequence of commuting variables and $\XX = (\xx_1,\xx_2,\dots)$ give an infinite sequence of noncommuting variables.   This paper follows the convention of \cite{intro}; variable names for bases are deliberately reused across $\sym$, $\nsym$, and $\qsym$.  Since this can occasionally cause confusion,  we will use ``standard'' (lowercase) type for bases of $\sym$, bold type for bases of $\nsym$, and capitalization for bases of $\qsym$ to make distinguishing them as easy as possible.  Repeatedly, we use $\mathbbm{1}_\mathcal{A}$, the indicator function that is $1$ if statement $\mathcal{A}$ is true and is $0$ if $\mathcal{A}$ is false.

Finally, throughout we give citations to theorems, indicating where they are stated in \cite{non}.  With few noted exceptions, the proofs in \cite{non} are distinct, and in many places definitions and theorems are reversed from the presentation below, due to the differences in what we take to be the definition of the space.

\section{Partitions and Compositions}
\noindent The vector spaces discussed here have bases that are naturally indexed by either partitions or compositions, so we begin there. 
\begin{definition} An infinite sequence of nonnegative integers $\alpha = (\alpha_1,\alpha_2,\alpha_3, \dots)$ is a \textbf{weak integer composition}, or simply a \textbf{weak composition}, of $n$ if $\sum \alpha_i = n$, denoted $\alpha \modelfix n$.  The $\alpha_i$ are the \textbf{parts} of $\alpha$, and the \textbf{size} of $\alpha$ is $n$, written $|\alpha| = n$.    \end{definition}

\begin{definition}
A \textbf{strong integer composition} of positive integer $n$, or often just simply a \textbf{composition} of $n$, is a finite sequence of positive integers $\alpha = (\alpha_1,\alpha_2,\dots,\alpha_{\ell(\alpha)})$ summing to $n$, denoted $\alpha \modelsstrong n$.  In this case, $\ell(\alpha)$ is the \textbf{length} of $\alpha$, and $\alpha$ has \textbf{last part} $\alpha_{\ell(\alpha)}$. The length of the empty composition is taken to be 0. \end{definition} 
\noindent We also write $\text{strong}(\alpha)$ to denote the composition attained by removing all zeros from the weak composition $\alpha$.  There is a classical bijection between strong compositions of $n$ and subsets of $[n-1] = \{1,2,\dots,n-1\}$.  In particular, if $\alpha=(\alpha_1,\alpha_2,\dots,\alpha_{\ell(\alpha)})$, we say that $$\set(\alpha)=\{\alpha_1,\alpha_1+\alpha_2,\dots,\alpha_1+\alpha_2+\cdots +\alpha_{\ell(\alpha)-1}\}.$$
There are three well-known involutions on the set of strong compositions (and their associated sets).  Consider $\alpha = (\alpha_1,\alpha_2,\dots,\alpha_{\ell(\alpha)}) \modelsstrong n$ and its associated set, $\text{set}(\alpha) \subseteq [n-1]$.
\begin{definition} The \textbf{reverse} of $\alpha$ is $\alpha^r = (\alpha_{\ell(\alpha)},\alpha_{\ell(\alpha)-1},\dots,\alpha_1).$  
\end{definition}
\begin{definition} If $A \subseteq [n-1]$, let $A^c = [n-1] \setminus A$, the complement of the set $A$ in $[n-1]$.  Then the \textbf{complement} of $\alpha$ is the composition 
$\alpha^c = \text{set}^{-1}(\text{set}(\alpha)^c).$
\end{definition}
\begin{definition}
The \textbf{transpose} of $\alpha$, written $\alpha^t$, is the composition obtained by applying the two previous involutions: $\alpha^t = (\alpha^r)^c = (\alpha^c)^r.$  
\end{definition}
\noindent It is not hard to check that composing any two of these distinct maps yields the third.  We will also repeatedly use the fact that for any (nonempty) strong composition $\alpha$, it is true that $\ell(\alpha) + \ell(\alpha^c) - 1 = |\alpha|$.

\begin{example} If $\alpha = (2,3,2,1) \modelsstrong 8$, then 
\begin{itemize}
\item $\alpha^r = (1,2,3,2);$
\item $\alpha^c = \text{set}^{-1}(\text{set}(\alpha)^c) = \text{set}^{-1}(\{2,5,7\}^c) = \text{set}^{-1}(\{1,3,4,6\}) = (1,2,1,2,2);$
\item $\alpha^t = (\alpha^c)^r = (1,2,1,2,2)^r = (2,2,1,2,1).$
\end{itemize}
Also see that $\ell(\alpha) + \ell(\alpha^c) - 1 = 4 + 5 - 1 = 8 = |\alpha|$.
\end{example}

\begin{definition} \label{def:refinement}
If $\alpha = (\alpha_1,\dots,\alpha_{\ell(\alpha)})$ and  $\beta = (\beta_1,\dots,\beta_{\ell(\beta)})$ are strong compositions of $n$, then \textbf{$\beta$ is a refinement of $\alpha$}, or \textbf{$\beta$ refines $\alpha$}, denoted $\beta \preceq \alpha$, if there exists an integer sequence $0=j_0<j_1<j_2<\cdots <j_{\ell(\alpha)}=\ell(\beta)$ such that for each $i \in \{1,2,\dots,\ell(\alpha)\}$,
$$\alpha_i = \beta_{j_{i-1}+1}+\beta_{j_{i-1}+2}+\cdots +\beta_{j_i}.$$ 
 That is, each part of $\alpha$ can be obtained by summing consecutive parts of $\beta$.  Equivalently, $\beta \preceq \alpha$ if and only if $\set(\beta)\supseteq \set(\alpha).$  In this context, let $\beta^{(i)}$ denote the subcomposition of $\beta$ which sums to $\alpha_i$ for $i =1,\dots,\ell(\alpha)$ so that $|\beta^{(i)}| = \alpha_i$.
\end{definition}
\noindent We note that the direction of refinement ($\preceq$) is not consistent across the literature in this area, and in particular, this work uses the opposite convention of \cite{non} and follows that of \cite{intro}.  
\begin{example} Both $\alpha = (1,6,3,4) \modelsstrong 13$ and $\beta = (1,3,1,2,2,1,1,3) \modelsstrong 13$.  See that $\beta \preceq \alpha$ with $\beta^{(1)} = (1),$ $\beta^{(2)} = (3,1,2)$, $\beta^{(3)} = (2,1)$, and $\beta^{(4)} = (1,3)$.
\end{example}

\noindent Since strong compositions of $n$ are in bijection with subsets of $[n-1]$, they inherit the same M\"{o}bius function, $\mu$.  We will repeatedly use the following basic facts:
\begin{itemize}
    \item For $S,T\subseteq [n-1]$, $\mu(S,T)=(-1)^{|S|-|T|}$.
    \item For $S,T\subseteq [n-1]$, $\sum_{S\subseteq U\subseteq T}\mu(U,T)=\mathbbm{1}_{S=T}$.
    \item (M\"{o}bius Inversion.) If $K$ is a commutative ring and $f,g:[n-1]\rightarrow K$, then 
    $$g(T)=\sum_{S\subseteq T}f(S)\,\,\,\text{ for all } T\subseteq [n-1]$$ if and only if $$f(T)=\sum_{S\subseteq T} g(S)\mu(S,T)\,\,\,\text{ for all } T\subseteq [n-1].$$
\end{itemize}

\begin{definition}
The \textbf{sort} of $\alpha$, $\text{sort}(\alpha)$, is the composition obtained by rewriting $\alpha$ in weakly decreasing order.
\end{definition}
\begin{definition}
An \textbf{integer partition} of $n$, or simply a \textbf{partition} of $n$, is a composition $\lambda$ of $n$ for which $\text{sort}(\lambda) = \lambda$, denoted $\lambda \vdash n$. 
 \end{definition}
\noindent Where the order of the parts of a composition $\alpha$ is immaterial, such as when $\alpha$ is a partition, we may write $\alpha = (1^{m_1(\alpha)}2^{m_2(\alpha)} \cdots n^{m_n(\alpha)})$, where 
$m_i(\alpha)$ gives the number of parts of size $i$ occurring in $\alpha$.  We may also use partial exponentiation to write certain compositions, e.g., $(1^k, n-k)$ will represent $(1,1,\dots,1,n-k) \models n$.

\section{Three Rings}

\noindent We consider three (graded) rings in this paper: the ring of symmetric functions, $\sym$, the ring of quasisymmetric functions, $\textnormal{QSym}$, and the ring of noncommutative symmetric functions, $\textbf{NSym}$.  Our goal in this section is to briefly introduce all three spaces, with an emphasis on $\textbf{NSym}$.  The unfamiliar reader may wish to consult \cite{Macdonald} or Chapter 7 of \cite{ECII} to learn more about the fundamentals of symmetric function theory, and \cite{intro} for an introduction to quasisymmetric functions.  There is not, to our knowledge, a well-known text covering noncommutative symmetric functions in detail.  A series of papers exploring noncommutative symmetric functions and their significance include \cite{non}, \cite{krob1997noncommutative}, \cite{duchamp1997noncommutative}, \cite{NSYMIV}, and \cite{krob1999noncommutative}.  These comprise perhaps the best known introduction to this area.

\begin{notation}
Let $X$ denote an infinite sequence of commuting variables $(x_1,x_2,\dots)$.  Given a composition $\alpha = (\alpha_1,\alpha_2,\dots)$, we write $x^\alpha$ to denote the monomial $x_1^{\alpha_1}x_2^{\alpha_2}\cdots$.  If $f\in \mathbb{C}[[x_1,x_2,\dots]]$ is a power series of bounded total degree, we write $f|_{x^\alpha}$ to denote the coefficient of $x^\alpha$ in the expansion of $f$ into monomials.  Any such $f$ is \textbf{homogeneous of degree $n$} if each monomial $x^\alpha$ appearing in $f$ is such that $|\alpha| = n$. 
\end{notation}

\begin{definition}
Let $f \in \mathbb{C}[[x_1,x_2,\dots]]$ be a power series of bounded total degree.  Then $f$ is \textbf{symmetric} if for all integers $k>0$,  all compositions $\alpha=(\alpha_1,\dots,\alpha_k)$, and all lists $(i_1,i_2,\dots, i_k)$ of distinct positive integers, $$f|_{x_1^{\alpha_1}x_2^{\alpha_2}\cdots x_k^{\alpha_k}}=f|_{x_{i_1}^{\alpha_1}x_{i_2}^{\alpha_2}\cdots x_{i_k}^{\alpha_k}}.$$
Let $\sym^n$ be the set of all symmetric functions homogeneous of degree $n$.  Then the ring of symmetric functions is $$\sym=\bigoplus_{n\geq 0} \sym^n.$$\end{definition}
\begin{example}\label{symfunc}  The power series
$$x_1^2x_2 + x_1^2x_3 + x_1^2x_4 + \cdots + x_1x_2^2 + x_2^2x_3 + x_2^2x_4 + \cdots \in \sym^3.$$
\end{example}

\begin{definition} Similarly, $f$ is \textbf{quasisymmetric} if for all integers $k>0$,  all compositions $(\alpha_1,\dots,\alpha_k)$, and all increasing lists $i_1<i_2<\dots< i_k$ of distinct positive integers, $$f|_{x_1^{\alpha_1}x_2^{\alpha_2}\cdots x_k^{\alpha_k}}=f|_{x_{i_1}^{\alpha_1}x_{i_2}^{\alpha_2}\cdots x_{i_k}^{\alpha_k}}.$$
Let $\textnormal{QSym}^n$ be the set of all quasisymmetric functions homogeneous of degree $n$.  Then the ring of quasisymmetric functions is $$\textnormal{QSym}=\bigoplus_{n\geq 0} \textnormal{QSym}^n.$$
\end{definition}
\begin{example} \label{quasiex} Both of the following power series are quasisymmetric functions in $\qsym^3$. 
$$x_1^2x_2 + x_1^2x_3 + \cdots + x_2^2x_3 + x_2^2x_4 + \cdots$$ 
$$x_1x_2^2 + x_1x_3^2 + \cdots + x_2x_3^2 + x_2x_4^2 + \cdots$$
\end{example}
\noindent Bases for $\sym^n$ are generally indexed by partitions of $n$, while bases for $\textnormal{QSym}^n$ are indexed by strong compositions of $n$ (or subsets of $[n-1]$).

\begin{definition} Let $I=(i_1,i_2,\dots, i_k)$ be a sequence of nonnegative integers.  Then the \textbf{descent set} of $I$ is 
$$\des(I)=\{j~|~i_j>i_{j+1}\}.$$
\end{definition}

\begin{definition}\label{def:Nsymbyalph} Let $\nsym$ be the subset of $ \mathbb{C}[[\xx_1,\xx_2,...]]$ of all power series $f$ of bounded total degree with the following property:  for all integers $k>0$ and all sequences of positive integers $I=(i_1,\dots, i_k)$ and $J=(j_1, \dots, j_k)$ such that  $\des(I)=\des(J)$,  $$f|_{\xx_{i_1}\cdots \xx_{i_k}}=f|_{\xx_{j_1}\cdots \xx_{j_k}}.$$
\end{definition}
\begin{example} We give the all the terms below which involve only the variables $\xx_1$ and $\xx_2$ for a particular element of $\nsym^3$ :
$$\xx_1^3 + \xx_1^2\xx_2 - \xx_1\xx_2\xx_1 + \xx_1\xx_2^2 + 2\xx_2\xx_1^2 + 2\xx_2\xx_1\xx_2 - \xx_2^2\xx_1 + \xx_2^3+\cdots$$
\end{example} 

\noindent Below, in Corollary \ref{cor:free}, we will show that the set of noncommutative elementary symmetric functions freely generate this space, giving a concrete realization of the space originally defined more abstractly by Gelfand et al.\ in \cite{non}.
Although this characterization is immediate from work in \cite{non}, as we will see below, the literature does not appear to include an explicit characterization of the overall space as above to our knowledge.   

\begin{remark*} For any $n \in \mathbb{N}$, the elements of $\nsym$ are constant under the standard paired parenthesis action of $\mathfrak{S}_n$, studied for example by Lascoux and Leclerc in \cite{lascoux2002plactic}, and thus satisfy a reasonable analogue of being ``symmetric.''  See \cite{non}, Proposition 7.17 for the proof, although it is immediate from the fact that one of the bases of $\nsym$ is a noncommutative version of a subset of the skew Schur functions sometimes referred to as the ribbon Schur functions.  (See just after Example \ref{originalribbonex}.)  Not all power series $f$ of bounded total degree in  $ \mathbb{C}[[\xx_1,\xx_2,\dots]]$ which are closed under this action are in $\nsym$, however, so one should be careful not to take this as a definition.  (See Example \ref{notnon} below.)

It is also worth noting that the noncommutative symmetric functions defined here are completely distinct from ``Symmetric Functions in Noncommuting Variables,'' developed more recently, which are defined by being fixed under the more standard $\mathfrak{S}_n$ action which simply permutes the indices of the noncommutative variables.  (See Rosas and Sagan \cite{bs}.)  \end{remark*} 

\begin{example} \label{notnon}
For example, $f=\sum_{i\geq 1}\xx_i^2$ is fixed under the standard action of the symmetric group, and thus is a symmetric function in noncommuting variables.  It is also closed under the paired parentheses action, and yet, it is not in $\nsym$, since, for example the coefficient of $\xx_1^2$ is not equal to the coefficient of $\xx_1\xx_2$ in $f$, although the descent set of $(1,1)$ is the same as the descent set of $(1,2)$.
\end{example}

\section{Bases of \texorpdfstring{$\sym^n$}{Sym} and \texorpdfstring{$\qsym^n$}{QSym}}
\noindent There are a number of well-known bases for each of the three spaces defined above.  In this section, we will briefly cover some of the relevant ones for $\sym^n$ and $\qsym^n$, before moving on to a more careful study of the bases of $\nsym^n$ and their relations. 

\subsection{Well-Known Bases of \texorpdfstring{$\sym^n$}{Sym}}
 \begin{definition} For $\lambda\vdash n$, the  \textbf{monomial symmetric function} (associated to $\lambda \vdash n$) is 
$$m_{\lambda} = \sum_{\substack{\alpha \models_w |\lambda| \\\text{sort}(\alpha) = \lambda}} x^\alpha.$$
\end{definition}

\noindent The monomial symmetric function associated to $\lambda \vdash n$ is minimal in the following sense: if $f\in\sym$ such that $f|_{x^\lambda}=1$, then the support of $f$ contains the support of $m_\lambda$, i.e., $$\{\alpha\modelfix n : \left. f\right|_{x^\alpha}\neq 0\}\supseteq \{\alpha\modelfix n : \left. m_\lambda\right|_{x^\alpha}\neq 0\}.$$
It is easy to see $\{m_\lambda\}_{\lambda \vdash n}$ is a basis of $\sym^n$.

\begin{definition}
 For $n \in \mathbb{N}$, the $n^{\text{th}}$ \textbf{elementary symmetric function} is 
$$e_n = \sum_{1\leq i_1 < i_2 < \cdots < i_n} x_{i_1}x_{i_2}\cdots x_{i_n},$$
with $e_0 = 1$.  Then, the elementary symmetric function associated to the partition $\lambda = (\lambda_1,\lambda_2,...,\lambda_{\ell(\lambda)})$ is defined multiplicatively: $e_\lambda = e_{\lambda_1}e_{\lambda_2}\cdots e_{\lambda_{\ell(\lambda)}}$.  The Fundamental Theorem of Symmetric Polynomials gives that $\{e_{\lambda}\}_{\lambda\vdash n}$ is a basis for $\sym^n$.
\end{definition}

\begin{definition} For $n \in \mathbb{N}$, the ($n^{\text{th}}$) \textbf{complete homogeneous symmetric function} is 
$$h_n = \sum_{1\le i_1 \le i_2 \le \cdots \le i_n} x_{i_1}x_{i_2}\cdots x_{i_n},$$
with $h_0 = 1$.  The complete homogeneous symmetric functions are also defined multiplicatively: $h_\lambda = h_{\lambda_1}h_{\lambda_2}\cdots h_{\lambda_{\ell(\lambda)}}$.  The set $\{h_{\lambda}\}_{\lambda\vdash n}$ is also a basis for $\sym^n$.
\end{definition}

\begin{definition}
   For $n \in \mathbb{N}$, the $n^\text{th}$ \textbf{power sum symmetric function} is
$$p_n = \sum_{i \in \mathbb{Z}^+} x_i^n,$$
with $p_0 = 1$.  Once more, $p_\lambda = p_{\lambda_1}p_{\lambda_2}\cdots p_{\lambda_{\ell(\lambda)}}$, and $\{p_{\lambda}\}_{\lambda\vdash n}$ is a basis for $\sym^n$.
\end{definition}
\noindent One additional important basis of $\sym$, the Schur functions, can be defined using semistandard Young tableaux, which we describe next.
\begin{definition}
 The \textbf{Young diagram of shape $\lambda = (\lambda_1,\lambda_2,...,\lambda_{\ell(\lambda)}) \vdash n$} is the left-justified array of $n$ boxes with $\lambda_i$ boxes in its $i^{\text{th}}$ row from the bottom (adopting French notation), $1 \le i \le \ell(\lambda)$.  In particular, we assume the box in the $i^{\text{th}}$ row and $j^{\text{th}}$ column has its upper-right-hand corner at the integer lattice point $(i,j)$.
\end{definition}

\begin{definition} A \textbf{filling} of a Young diagram of shape $\lambda \vdash n$ is one where each of its boxes is filled with a positive integer.  The resulting filled Young diagram is called a \textbf{Young tableau}.

If in Young tableau $T$ the integers are both weakly increasing in the rows from left-to-right and strictly increasing in the columns from bottom-to-top, we call $T$ a \textbf{semi-standard Young tableau}.  The set of all semi-standard Young tableaux of shape $\lambda$ is denoted $\text{SSYT}(\lambda)$.  A \textbf{standard Young tableau} is a semistandard Young tableau in which each of the integers $\{1,2,\dots,|\lambda|\}$ occurs exactly once.
\end{definition}

\begin{definition}
    
 The \textbf{type} or \textbf{content} of a Young tableau $T$, $\text{type}(T)$, is the composition $(m_1,m_2,\dots,m_n) \modelfix|\lambda|$ 
where $m_i$ is the number of $i$'s appearing in $T$.  So in Example \ref{tableau} below, $\text{type}(T) = (2,2,3)$.
\end{definition}

\noindent We can now define the final well-known basis of $\sym$, which interpolates between the elementary and complete homogeneous symmetric functions.

\begin{definition} The \textbf{Schur function} associated to partition $\lambda \vdash n$ is
$$s_\lambda = \sum_{T \in \text{SSYT}(\lambda)} x^{\text{type}(T)}.$$
\end{definition}
\begin{definition} Let $\lambda$ be a partition of $n$. If the Young diagram of shape $\lambda$ is flipped about the southwest-northeast diagonal, the resulting Young diagram is of shape $\lambda^t$, the \textbf{transpose} or \textbf{conjugate} partition of $\lambda$.  Note that $(\lambda^t)^t = \lambda$.
\end{definition}
\begin{example} \label{tableau} The following Young tableau is semistandard, but not standard.
$$T = \begin{ytableau}
3\\
2 & 3\\
1 & 1 & 2 & 3
\end{ytableau} \in ~~\text{SSYT}(4,2,1).$$

\noindent The conjugate partition of $\lambda = (4,2,1)$ is $\lambda^t = (3,2,1,1)$.
\end{example}

\subsubsection{Involutions and the Hall inner product on \texorpdfstring{$\sym$}{Sym}}
Since $\langle e_n \rangle_{n \in \mathbb{N}}$ generates all of $\sym$, the following defines a homomorphism on $\sym$:
\begin{definition}
    Define the endomorphism $\omega: \sym \rightarrow \sym$ by setting $\omega(e_n) = h_n$ for all $n \in \mathbb{N}$. 
\end{definition} 
Then we have the following well-known theorem:
\begin{theorem} For any partition $\lambda$,
\begin{align*}\omega(e_\lambda) &= h_\lambda;\\
\omega(h_\lambda) &= e_\lambda;\\
\omega(p_\lambda) &= (-1)^{|\lambda| - \ell(\lambda)}p_\lambda;\\
\omega(s_\lambda) &= s_{\lambda^t}.
\end{align*} 
\end{theorem}
\noindent By the first two lines, $\omega$ is an involution, and since $\{m_\lambda\}$ is a basis of $\sym$, the set $\{\omega(m_\lambda)\}$ must form a basis of $\sym$ as well.
\begin{definition}
The \textbf{forgotten symmetric function} associated to $\lambda$ is defined to be $f_\lambda = \omega(m_\lambda)$.\footnote{The forgotten symmetric functions are sometimes defined by $f_{\lambda} = (-1)^{|\lambda|} \omega(m_\lambda)$. (See Doubilet \cite{Doubilet}).}  Generally, the forgotten symmetric functions are defined indirectly via the map $\omega$ or the Hall inner product and duality.
\end{definition}
\begin{definition}
 The \textbf{Hall inner product}, $\langle \cdot, \cdot \rangle: \sym \times \sym \rightarrow \mathbb{C}$, is determined by 
\begin{eqnarray} \label{Hallinn} \langle m_\lambda, h_\mu \rangle = \mathbbm{1}_{\lambda = \mu}.\end{eqnarray} 
\end{definition}

\begin{theorem} The forgotten symmetric functions are dual to the elementary symmetric functions, and the Schur functions form an orthonormal basis of $\sym$ under the Hall inner product.  That is,
 \begin{align*}\langle f_\lambda, e_\mu \rangle =\langle s_\lambda, s_\mu \rangle = \mathbbm{1}_{\lambda = \mu}.  
 \end{align*}
 Furthermore, it is also true that 
$$\langle p_\lambda, p_\mu \rangle = z_\lambda \cdot \mathbbm{1}_{\lambda = \mu},$$ 
where
$$z_\lambda = \prod_{i=1}^{\ell(\lambda)} i^{m_i(\lambda)}m_i(\lambda)!$$ is the well-known combinatorial coefficient that measures the size of the centralizer of any symmetric group element having cycle type $\lambda$.
\end{theorem}
\noindent It can be also be shown that $\omega$ is an isometry with respect to the Hall inner product, i.e., for any $g,g' \in \sym$, 
$$\langle \omega(g),\omega(g') \rangle = \langle g,g' \rangle.$$  

\noindent Figure \ref{Duality of Sym} summarizes the results we have provided pertaining to the relationships between the various bases of $\sym$ under $\omega$ and $\langle \cdot, \cdot \rangle$.

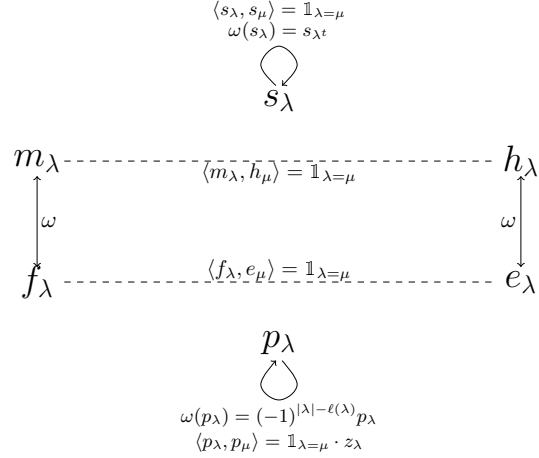
\begin{figure}
$$\begin{tikzpicture}
    \useasboundingbox (-3.5,-3) rectangle (3.5,3);
      \scope[transform canvas={scale=.8}]
\node at (0,2) {\Large{$s_\lambda$}};
\node at (-4,1) {\Large{$m_\lambda$}};
\node at (-4,-1) {\Large{$f_\lambda$}};
\node at (4,1) {\Large{$h_\lambda$}};
\node at (4,-1) {\Large{$e_\lambda$}};
\node at (0,-2) {\Large{$p_\lambda$}};

\node at (3.8,0) {\small{$\omega$}};
\node at (-3.8,0) {\small{$\omega$}};
\node at (0,-3.2) {\scriptsize{$\omega(p_\lambda) = (-1)^{|\lambda| - \ell(\lambda)}p_\lambda$}};
\node at (0,3.15) {\scriptsize{$\omega(s_\lambda) = s_{\lambda^t}$}};
\node at (0,-3.65) {\scriptsize{$\langle p_\lambda, p_\mu \rangle = \mathbbm{1}_{\lambda = \mu} \cdot z_\lambda$}};
\node at (0,3.5) {\scriptsize{$\langle s_\lambda, s_\mu \rangle = \mathbbm{1}_{\lambda = \mu}$}};
\node at (0,0.8) {\footnotesize{$\langle m_\lambda, h_\mu \rangle = \mathbbm{1}_{\lambda = \mu}$}};
\node at (0,-0.8) {\footnotesize{$\langle f_\lambda, e_\mu \rangle = \mathbbm{1}_{\lambda = \mu}$}};

\draw [-latex][<->] plot[smooth, tension=.7] coordinates {(-4,0.75) (-4,-0.75)};
\draw [-latex][<->] plot[smooth, tension=.7] coordinates {(4,-0.75) (4,0.75)};
\draw [-latex][->] plot[smooth, tension=.7] coordinates {(0.05,-2.3) (0.3,-2.7)(0,-2.95)(-0.3,-2.7)(-0.05,-2.3)};
\draw [-latex][->] plot[smooth, tension=.7] coordinates {(-0.05,2.25) (-0.3,2.6)(0,2.9)(0.3,2.6)(0.05,2.25)};
\draw [dashed][-] plot[smooth, tension=.7] coordinates {(-3.55,1)(3.55,1)};
\draw [dashed][-] plot[smooth, tension=.7] coordinates {(-3.55,-1)(3.55,-1)};
\endscope
\end{tikzpicture}$$
    \caption{~Images of $\omega$ and Self-Duality of $\sym$}
    \label{Duality of Sym}
\end{figure}

\subsection{Well-Known Bases of \texorpdfstring{$\qsym$}{QSym}}  As mentioned above, bases for $\qsym$ are indexed by strong compositions of $n$ or subsets of $[n-1]$.
A first basis for $\qsym$ is the following:
\begin{definition} For any strong composition $\alpha$, the
    \textbf{monomial quasisymmetric function} (associated to $\alpha \modelsstrong n$) is
$$M_\alpha = \sum_{\substack{\beta \models_w |\alpha| \\ \text{strong}(\beta) = \alpha}} x^\beta.$$
\end{definition}
\noindent The monomial quasisymmetric function associated to $\alpha \modelsstrong n$ is minimal in the following sense: if $f\in\qsym$ such that $f|_{x^\alpha}=1$, then the support of $f$ contains the support of $M_\alpha$, i.e., $$\{\beta\modelfix n : \left. f\right|_{x^\beta}\neq 0\}\supseteq \{\beta\modelfix n : \left. M_\alpha\right|_{x^\beta}\neq 0\}.$$    
\noindent For any $\lambda \vdash n$,
$$m_\lambda = \sum_{\substack{\alpha \modelsstrong n \\ \text{sort}(\alpha) = \lambda}} M_\alpha.$$

\noindent Another well-known basis of $\textnormal{QSym}^n$ is \textbf{Gessel's fundamental basis}, so called since it was defined by Gessel in \cite{Gessel}.  It is one of several bases of quasisymmetric functions considered a generalization of the Schur symmetric function basis.
\begin{definition} For $\alpha \modelsstrong n$, the \textbf{Gessel Fundamental quasisymmetric function} associated to $\alpha$ is
\begin{eqnarray*} \label{QSymFtoM}
F_\alpha = \sum_{\substack{i_1 \le i_2 \le \cdots \le i_n\\ k \in \text{set}(\alpha)\Rightarrow i_k < i_{k+1}}} x_{i_1}x_{i_2}\cdots x_{i_n}.
\end{eqnarray*}
 
\end{definition}

\noindent There are two other commonly studied quasisymmetric analogues of the Schur functions: the quasisymmetric Schur functions, first defined by Haglund et al.\ in \cite{haglund2011quasisymmetric}, and the dual immaculate functions, defined by duality to the immaculate noncommutative symmetric functions by Berg et al.\ in \cite{berg2014lift}.  Images of these bases under the automorphisms below are also sometimes studied.  We also mention for completeness that there are two quasisymmetric analogues of the power sums, which were first defined indirectly as dual bases, as described in the next section.

There are not one, but three natural analogues of $\omega$ defined on $\qsym$.  As in \cite{intro}, define $\rho, \psi, \omega: \textnormal{QSym} \rightarrow \textnormal{QSym}$, all automorphisms, by \begin{eqnarray}
\rho(F_\alpha) &=& F_{\alpha^r}; \label{QSymrho}
\\ \psi(F_\alpha) &=& F_{\alpha^c}; \label{QSympsi}
\\ \omega(F_\alpha) &=& F_{\alpha^t}. \label{QSymomega}
\end{eqnarray}
The names for these maps are not at all uniform in the literature, so here, as elsewhere, we follow the convention of \cite{intro}.

\section{Bases for \texorpdfstring{$\nsym$}{NSym}}
\noindent We begin with the basis for the space with the minimal number of terms, the noncommutative ribbon Schur functions.
\begin{definition} The \textbf{noncommutative ribbon Schur function} associated to $\alpha \modelsstrong n$ is
\begin{eqnarray*} \label{Symribbontox}
\nr_\alpha = \sum_{\substack{I=(i_1,i_2,\dots, i_n) \in (\mathbb{Z}^+)^n \text{ s.t. }\\ \des\left(I\right)=\set(\alpha)}} \xx_{i_1}\xx_{i_2}\cdots \xx_{i_n}.
\end{eqnarray*}
 \end{definition}
 \begin{theorem}[\cite{non}, section 4.4]
     $\{\nr_{\alpha}\}_{\alpha\modelsstrong n}$ is a basis for $\nsym^n$.
 \end{theorem}
 \begin{proof}
It is clear from the definition of the space  that $\nsym$ is the span of the noncommutative ribbon Schur functions.  To see they are independent, note that for every sequence of positive integers $I=(i_1,i_2,\dots,i_n)$, the monomial $\xx_{i_1}\xx_{i_2}\cdots \xx_{i_n}$ occurs uniquely  with positive coefficient in exactly one $\nr_\alpha$; in particular, where $\alpha=\set^{-1}(\des(I))$.
\end{proof}
\noindent The following theorem also follows easily from the definition.
\begin{theorem}\label{thm:ribbon_mult}  For strong compositions $\alpha=(\alpha_1,\alpha_2,\dots, \alpha_n)$ and $\beta=(\beta_1,\beta_2,\dots,\beta_k)$,
    $$\nr_{\alpha} \nr_{\beta}=\nr_{(\alpha_1,\dots, \alpha_n,\beta_1,\dots,\beta_k)}+\nr_{(\alpha_1,\dots,\alpha_{n-1}, \alpha_n+\beta_1,\beta_2,\dots,\beta_k)}$$
\end{theorem}
\noindent   The noncommutative ribbon Schur function $\nr_\alpha$ is a minimal noncommutive symmetric function in the following sense: if $\boldsymbol{f}\in\nsym$ such that $\boldsymbol{f}|_{{(\xx_{\ell(\alpha)})}^{\alpha_1}(\xx_{\ell(\alpha)-1})^{\alpha_2}\cdots ({\xx_1})^{\alpha_{\ell(\alpha)}}}=1$, then the support of $\boldsymbol{f}$ contains the support of $r_\alpha$.  That is, for all $n$,  $$\{(i_1,i_2,\dots,i_n) : \left. \boldsymbol{f}\right|_{\xx_{i_1}\xx_{i_2}\cdots \xx_{i_n}}\neq 0\}\supseteq \{(i_1,i_2,\dots,i_n) : \left. \nr_\alpha\right|_{\xx_{i_1}\xx_{i_2}\cdots \xx_{i_n}}\neq 0\}.$$ 

\noindent While the noncommutative ribbon Schur functions share a minimality condition with the monomial symmetric functions and the monomial quasisymmetric functions, they, like the fundamental quasisymmetric functions, are usually considered an analogue of the Schur functions for a number of reasons. 
 \begin{definition}
A \textbf{skew Young diagram} associated to $\lambda/ \mu$ is realized by taking a Young diagram, say of shape $\lambda$, and removing $\mu$, a Young diagram sitting inside it.  If the resulting skew Young diagram is connected and contains no $2 \times 2$ boxes, it is called a \textbf{ribbon Young diagram}.
\end{definition}
\begin{example}
The skew Young diagram $(3,2,1)/(1)$ is a ribbon Young diagram since it has no $2 \times 2$ boxes, as seen below.
$$
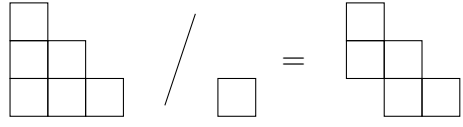
\begin{tikzpicture}
\draw  (0,0) rectangle (.5,.5);
\draw  (.5,0) rectangle (1,.5); 
\draw  (1,0) rectangle (1.5,.5);
\draw  (0,1) rectangle (0.5,1.5);
\draw  (0,.5) rectangle (.5,1);
\draw  (.5,.5) rectangle (1,1);
\draw (2.75,0) rectangle (3.25,0.5);
\node (v1) at (2,0) {};
\node (v2) at (2.5,1.5) {};
\draw  (v1) edge (v2);
\node at (3.75,0.7) {$=$};
\draw  (4.95,0) rectangle (5.45,0.5); 
\draw  (5.45,0) rectangle (5.95,0.5);
\draw  (4.45,1) rectangle (4.95,1.5);
\draw  (4.45,0.5) rectangle (4.95,1);
\draw  (4.95,0.5) rectangle (5.45,1);
\end{tikzpicture}$$
\end{example}
\noindent To stay consistent with \cite{non}, say that this is the ribbon of \textbf{shape} $(1,2,2)$, the lengths of the resulting rows when read from \textit{top-to-bottom}.

It is from these combinatorial objects that the basis above gets its name.  If the noncommutative ribbon Schur function $\boldsymbol{r}_\alpha$ is expanded in the noncommuting variables $\boldsymbol{x}$, the indices appearing on the monomials appearing (left-to-right) in $\boldsymbol{r}_\alpha$ fill in ribbons of shape $\alpha$ to yield \textbf{ribbon tableaux} of shape $\alpha$.
\begin{definition}
 A \textbf{ribbon tableau of shape $\alpha$} is a filling of the ribbon of shape $\alpha$ with positive integers that weakly increase across rows, left-to-right, and increase along columns, bottom-to-top.  That is, ribbon tableaux are just semistandard Young tableaux of ribbon shape.
\end{definition}
\begin{example} \label{originalribbonex} Let $\alpha = (1,2,2)$.  Then one of infinitely many ribbon tableaux of shape $\alpha$ is the one corresponding to the monomial $\boldsymbol{x}_3\boldsymbol{x}_1\boldsymbol{x}_2\boldsymbol{x}_1^2,$ depicted below.

$$
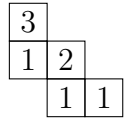
\begin{tikzpicture}[scale=0.5]
\draw  (1,0) rectangle (2,1); 
\draw  (2,0) rectangle (3,1);
\draw  (0,2) rectangle (1,3);
\draw  (0,1) rectangle (1,2);
\draw  (1,1) rectangle (2,2);
\node at (0.5,2.5) {3};
\node at (0.5,1.5) {1};
\node at (1.5,1.5) {2};
\node at (1.5,0.5) {1};
\node at (2.5,0.5) {1};
\end{tikzpicture}$$
\end{example}
\begin{definition} Let $\chi:\mathbb{C}[[\xx_1,\xx_2,\cdots]]\rightarrow \mathbb{C}[[x_1,x_2,\cdots]]$ be the ``forgetful'' function that sends the noncommutative variables to their commutative analogues: i.e.\ $\chi(\xx_i)=x_i$. 
\end{definition}
Then $\chi(\nr_\alpha)=r_\alpha$, the commutative ribbon function $\alpha$; it is also the skew Schur function of ribbon shape $\alpha$. The set $\{r_\alpha\}_{\alpha \models n}$ forms a spanning set of $\sym^n$, since for $\lambda\vdash n$, $h_\lambda=\sum_{\set(\alpha)\subseteq \set(\lambda)}r_\alpha$ and thus $\chi(\nsym)=\sym$.

With a generalization of the Hall inner product defined in \cite{non}, $\textnormal{QSym}$ and $\textbf{NSym}$ are dual spaces.  Define $\langle \cdot, \cdot \rangle: \textnormal{QSym} \times \textbf{NSym} \rightarrow \mathbb{C}$, where
\begin{eqnarray} \label{QNinnerproduct}
\langle F_\alpha, \nr_\beta \rangle = \mathbbm{1}_{\alpha = \beta}.
\end{eqnarray}

The combined results of Gessel in \cite{Gessel} and Malvenuto and Reutenauer in \cite{MalvenutoReutenauer} imply that $\qsym$ and $\nsym$ are dual with respect to this inner product, as first observed in \cite{non}.

There are three natural involutions $f$ on $\nsym$ that when composed with the forgetful map $\chi$ give $\chi\circ f=\omega\circ \chi$ that correspond to the three involutions on $\qsym$ mentioned above.

\begin{definition} Let $\boldsymbol{\rho},\boldsymbol{\psi}, \boldsymbol{\omega}:\textbf{NSym} \rightarrow \textbf{NSym}$ be linear transformations that satisfy the following:
\begin{eqnarray}
\boldsymbol{\rho}(\boldsymbol{r}_\alpha) &=& \boldsymbol{r}_{\alpha^r}  \label{NSymrho}
\\ \boldsymbol{\psi}(\boldsymbol{r}_\alpha) &=& \boldsymbol{r}_{\alpha^c}  \label{NSympsi}
\\ \boldsymbol{\omega}(\boldsymbol{r}_\alpha) &=& \boldsymbol{r}_{\alpha^t}  \label{NSymomega}
\end{eqnarray} 
\end{definition}

\noindent Here again we adopt the notation as in \cite{intro} and we note, just as they did, that $\boldsymbol{\psi}$ is an automorphism, and that $\boldsymbol{\rho}$ and $\boldsymbol{\omega}$ are anti-automorphisms. 
The following theorem follows easily from Theorem \ref{thm:ribbon_mult} and the above definitions. \begin{theorem} For any strong compositions $\alpha$ and $\beta$,
    \begin{align*}
 \boldsymbol{\rho}(\boldsymbol{r}_\alpha\boldsymbol{r}_\beta) = \boldsymbol{\rho}(\boldsymbol{r}_\beta)\boldsymbol{\rho}(\boldsymbol{r}_\alpha);
\\\boldsymbol{\psi}(\boldsymbol{r}_\alpha\boldsymbol{r}_\beta) = \boldsymbol{\psi}(\boldsymbol{r}_\alpha)\boldsymbol{\psi}(\boldsymbol{r}_\beta);
\\ \boldsymbol{\omega}(\boldsymbol{r}_\alpha\boldsymbol{r}_\beta) = \boldsymbol{\omega}(\boldsymbol{r}_\beta)\boldsymbol{\omega}(\boldsymbol{r}_\alpha). 
\end{align*} 
\end{theorem}

\noindent For any $f \in \textnormal{QSym}$ and $\boldsymbol{g} \in \textbf{NSym}$,
\begin{eqnarray} \label{InnProinvos}
\langle f, \boldsymbol{g} \rangle = \langle \rho(f),\boldsymbol{\rho}(\boldsymbol{g}) \rangle = \langle \psi(f),\boldsymbol{\psi}(\boldsymbol{g}) \rangle = \langle \omega(f),\boldsymbol{\omega}(\boldsymbol{g}) \rangle.
\end{eqnarray}

\begin{definition} Let $\nel_0=1$, $\nh_0=1$, and for any positive integer $n$, define
\begin{eqnarray}  
\nel_n = \sum_{i_1 > i_2 > \cdots > i_n\geq 1} \xx_{i_1}\xx_{i_2}\cdots \xx_{i_n}; \hspace{1cm} \nh_n = \sum_{1\le i_1 \le i_2 \le \cdots \le i_n} \xx_{i_1}\xx_{i_2} \cdots \xx_{i_n}.
\end{eqnarray}
Then $\nel_n$ is the $n^\text{th}$ \textbf{noncommutative elementary symmetric function} and $\nh_n$ is the $n^\text{th}$ \textbf{noncommutative homogeneous complete symmetric function}.  
For any $\alpha\modelsstrong n$, let

$$\nel_\alpha = \nel_{\alpha_1}\nel_{\alpha_2}\cdots \nel_{\alpha_{\ell(\alpha)}} \text{ and } \nh_\alpha = \nh_{\alpha_1} \nh_{\alpha_2}\cdots \nh_{\alpha_{\ell(\alpha)}}.$$
\end{definition}

\noindent It is easy to see the following from the definitions of $\nh_n$, $\nel_n$, and $\nr_\alpha$.
\begin{theorem} For all $n\geq 0$, and any $\alpha \modelsstrong n$,
$$\chi(\nel_n)=e_n\text{,  }\chi(\nh_n)=h_n\text{, and }\chi(\nr_\alpha)=r_\alpha.$$
\end{theorem}

\begin{theorem}[\cite{non}, section 4.4 and 4.7]\label{thm:ehtor}
\begin{align}
&\boldsymbol{h}_\alpha = \sum_{\beta \succeq \alpha} \boldsymbol{r}_\beta  &&\boldsymbol{r}_\alpha = \sum_{\beta \succeq \alpha} (-1)^{\ell(\alpha) - \ell(\beta)}\boldsymbol{h}_\beta \label{thm:hrCOB}
\\[1em]&\boldsymbol{e}_\alpha = \sum_{\beta^c \succeq \alpha} \boldsymbol{r}_\beta  &&\boldsymbol{r}_\alpha = \sum_{\beta \succeq \alpha^c} (-1)^{\ell(\alpha^c) - \ell(\beta)}\boldsymbol{e}_\beta \label{thm:erCOB}
\end{align}
\end{theorem}
\begin{proof} For $\alpha \modelsstrong n$, and any $I=(i_1,i_2,\dots,i_n) \in \mathbb{Z}^n$,
$$\left. \boldsymbol{h}_{\alpha}\right|_{\xx_{i_1}\xx_{i_2}\cdots \xx_{i_n}}=\mathbbm{1}_{\set(\alpha)\supseteq \des(I)},$$ giving the left-hand side of (\ref{thm:hrCOB}).  Similarly,    $$\left. \boldsymbol{e}_{\alpha}\right|_{\xx_{i_1}\xx_{i_2}\cdots \xx_{i_n}}=\mathbbm{1}_{\set(\alpha)^c\subseteq \des(I)},$$ giving the left-hand side of (\ref{thm:erCOB}).  Then, as observed in \cite{non} the equations on the right are an application of M\"{o}bius Inversion, applied to the Boolean algebra.
\end{proof}
\begin{corollary}[\cite{non}, p.\ 16]\label{cor:free}
    Both $\{\nel_\alpha\}_{\alpha\modelsstrong n}$ and $\{\nh_\alpha\}_{ \alpha \modelsstrong n}$ are bases of $\nsym^n$.  In particular, $\nsym$ is generated freely by $\{\nel_n\}_{n\in \mathbb{N}}$ or $\{\nh_n\}_{n \in \mathbb{N}}$. 
\end{corollary}

\noindent Thus Definition \ref{def:Nsymbyalph} above, which gives $\nsym$ as a subspace of $\mathbb{C}[[\xx_1,\xx_2,\cdots]]$, defines a space which is isomorphic to the more abstractly defined space of Gelfand et al.\ as described in Definition \ref{def:Nsymbyformal}, justifying the use of the same name.  Moreover, 
\begin{corollary}[\cite{non}, p.\ 15 and p.\ 19] For any strong compositions $\alpha$ and $\beta$,
   
\begin{eqnarray} \label{NSyminvos}
\boldsymbol{\rho}(\boldsymbol{h}_\alpha) = \boldsymbol{h}_{\alpha^r}, \hspace{1cm} \boldsymbol{\psi}(\boldsymbol{h}_\alpha) = \boldsymbol{e}_\alpha, \hspace{1cm} \boldsymbol{\omega}(\boldsymbol{h}_\alpha) = \boldsymbol{e}_{\alpha^r},
\end{eqnarray}
and 
\begin{align} \label{QNinnerproduct2}
\langle M_\alpha, \boldsymbol{h}_\beta \rangle = \mathbbm{1}_{\alpha = \beta}.
\end{align}
\end{corollary}

\begin{theorem}[\cite{non}, Section 4.1]\label{thm:heCOB} For $n\geq 1$, and any $\beta \modelsstrong n$,
    $$\nh_\beta=\sum_{\alpha\preceq\beta }(-1)^{(\ell(\alpha)-|\beta|)}\nel_\alpha.$$
\end{theorem}
\begin{proof}
\begin{align*}
\sum_{\alpha\preceq\beta}(-1)^{(\ell(\alpha)-|\beta|)}\nel_\alpha&=\sum_{\alpha\succeq\beta^c}(-1)^{(\ell(\alpha)-1)}\nel_{\alpha^c}&&\alpha\rightarrow \alpha^c\\
    &=\sum_{\alpha\succeq \beta^c}(-1)^{(\ell(\alpha)-1)}\sum_{\gamma\preceq\alpha}r_\gamma&&\text{by (\ref{thm:erCOB})}\\
    &=\sum_{\gamma\modelsstrong|\beta|}\nr_\gamma\sum_{\substack{\alpha\succeq\gamma\\\alpha\succeq \beta^c}}(-1)^{(\ell(\alpha)-1)}\\
    &=\sum_{\gamma\modelsstrong|\beta|}\nr_\gamma\sum_{\set(\alpha)\subseteq\set(\gamma)\cap\set(\beta)^c}(-1)^{|\set(\alpha))|}\\
    &=\sum_{\gamma \succeq\beta}\nr_\gamma&&\text{by M\"{o}bius function properties}\\
    &=    \nh_\beta. &&\text{by (\ref{thm:hrCOB})}
\end{align*}
The second-to-last equality follows from properties of the M\"{o}bius function on the Boolean algebra, since the sum will be nonzero unless $\set(\gamma)\cap\set(\beta)^c=\emptyset.$
\end{proof}

\noindent As in $\sym$, the generating series of the noncommutative elementary and complete homogeneous symmetric functions are particularly nice:
\begin{eqnarray} \label{NSymgenfunc}
\textbf{E}(t) = \sum_{n \in \mathbb{N}} \boldsymbol{e}_nt^n=\prod_{i\geq 1}^\leftarrow (1+\xx_it) \hspace{0.3cm} \text{and} \hspace{0.3cm} \textbf{H}(t) = \sum_{n \in \mathbb{N}} \boldsymbol{h}_nt^n=\prod_{i\geq 1}^\rightarrow \frac{1}{(1-\xx_it)}.
\end{eqnarray}
Here, we must take $t$ to be a formal variable which commutes with $\xx_i$ for all $i$. It is easy to see that, as in the commutative case, \begin{align} \label{eq:HE}\textbf{E}(-t)\textbf{H}(t) = 1=\textbf{H}(t)\textbf{E}(-t).\end{align} This is taken as the defining relation for the noncommutative homogeneous basis in \cite{non}.  As observed there, with this relation it is immediate from comparing the coefficients of $t^n$ in (\ref{eq:HE}) that
\begin{theorem}[\cite{non}, Proposition 3.3] For any positive integer $n$,\begin{eqnarray} \label{NSymeandh}
\sum_{i=0}^n (-1)^{n-i} \boldsymbol{e}_i\boldsymbol{h}_{n-i} = 0=\sum_{i=0}^n (-1)^{n-i} \boldsymbol{h}_i\boldsymbol{e}_{n-i}.
\end{eqnarray}
\end{theorem}

\noindent Last, we turn our attention to the noncommutative power sums.  The reader expecting a similarly simple definition of $\np_n$, analogous to that of $\nel_n$ and $\nh_n$, will be disappointed.
In particular, as mentioned above for $k>1$, $$\sum_{i\geq 1}\xx_i^k\notin \nsym.$$
Thus, if one wishes to define $\np_n\in \nsym$ such that $\chi(\np_n)=p_n$, it must be that $\np_n$ has both positive and negative terms, when written as a sum of monomials. There is not a unique such $\np_n$ in $\nsym$; Gelfand et al.\ in \cite{non} define two noncommutative analogues of the power sums, $\nphi_n$ and $\npsi_n$.  While they originally define these bases based on their relation to the noncommutative complete homogeneous basis, we begin with their expansion in the noncommutative ribbon basis, since this allows us to easily read off their definition in terms of monomials.

\begin{definition}\label{def:phipsi_n}  Let $\npsi_0=1$, $\nphi_0=1$, and for $n\geq 1$, let
    \begin{align}\npsi_n
    &=\sum_{I=(i_1,\dots,i_n)\in A_n} (-1)^{k(I)-1}\xx_{i_1}\xx_{i_2}\cdots \xx_{i_n}\\&=\sum\limits_{k=0}^{n-1} (-1)^k \nr_{1^k(n-k)}\end{align}
    and \begin{align}\nphi_n
    &=\sum\limits_{I=(i_1,i_2,\dots,i_n)\in (\mathbb{Z}^+)^n} \frac{(-1)^{|\des(I)|}}{\binom{n-1}{|\des(I)|}} \xx_{i_1}\xx_{i_2}\cdots \xx_{i_n}\\&= \sum\limits_{\alpha\modelsstrong n} \frac{(-1)^{\ell(\alpha)-1}}{\binom{n-1}{\ell(\alpha)-1}} \nr_\alpha, \end{align}
  where \begin{align*}A_n=&\left\{\left.({i_1},\dots, {i_n})\in  (\mathbb{Z}^+)^n ~\right|~\exists k \text{ s.t.\ }\right.\\&\hspace{2cm}\left. 1 \leq k\leq n \text{ and }i_1>i_2>\cdots>i_{k-1}>i_{k}\leq i_{k+1}\leq \cdots \le i_n\right\},\end{align*} and where $k(I)$ is the unique $k$ satisfying the condition required for $I=({i_1},\dots,{i_n}) \in A_n$.

For any strong composition $\alpha=(\alpha_1,\alpha_2,\dots)$,
$\boldsymbol{\psi}_\alpha = \boldsymbol{\psi}_{\alpha_1}\boldsymbol{\psi}_{\alpha_2} \cdots$ and $\boldsymbol{\phi}_\alpha = \boldsymbol{\phi}_{\alpha_1}\boldsymbol{\phi}_{\alpha_2} \cdots$. Then $\{\boldsymbol{\psi}_{\alpha}\}_{ \alpha \modelsstrong n}$ and $\{\boldsymbol{\phi}_\alpha\}_{ \alpha \modelsstrong n}$ are the \textbf{noncommutative power sums of the 1st} and \textbf{2nd kinds} (respectively).

\end{definition}
\begin{example}
   \begin{align*}\psi_3&=\xx_{1}^{3} + \xx_{1}^{2}\xx_{2} + \xx_{1}^{2}\xx_{3} + \xx_{1}\xx_{2}^{2} + \xx_{1}\xx_{2}\xx_{3} + \xx_{1}\xx_{3}^{2} - \xx_{2}\xx_{1}^{2} - \xx_{2}\xx_{1}\xx_{2}\\&\hspace{1cm}  - \xx_{2}\xx_{1}\xx_{3}+ \xx_{2}^{3} + \xx_{2}^{2}\xx_{3} + \xx_{2}\xx_{3}^{2} - \xx_{3}\xx_{1}^{2} - \xx_{3}\xx_{1}\xx_{2} - \xx_{3}\xx_{1}\xx_{3} \\&\hspace{1.5cm}+ \xx_{3}\xx_{2}\xx_{1}  - \xx_{3}\xx_{2}^{2}- \xx_{3}\xx_{2}\xx_{3} + \xx_{3}^{3}+\cdots\end{align*}
    and
    \begin{align*}\phi_3&=\xx_{1}^{3} + \xx_{1}^{2}\xx_{2} + \xx_{1}^{2}\xx_{3} - \frac{1}{2}\xx_{1}\xx_{2}\xx_{1} + \xx_{1}\xx_{2}^{2} + \xx_{1}\xx_{2}\xx_{3} - \frac{1}{2}\xx_{1}\xx_{3}\xx_{1} - \frac{1}{2}\xx_{1}\xx_{3}\xx_{2} + \xx_{1}\xx_{3}^{2} \\&\hspace{1cm}- \frac{1}{2}\xx_{2}\xx_{1}^{2} - \frac{1}{2}\xx_{2}\xx_{1}\xx_{2} - \frac{1}{2}\xx_{2}\xx_{1}\xx_{3} - \frac{1}{2}\xx_{2}^{2}\xx_{1} + \xx_{2}^{3} + \xx_{2}^{2}\xx_{3} - \frac{1}{2}\xx_{2}\xx_{3}\xx_{1} \\&\hspace{1cm} - \frac{1}{2}\xx_{2}\xx_{3}\xx_{2} + \xx_{2}\xx_{3}^{2}- \frac{1}{2}\xx_{3}\xx_{1}^{2} - \frac{1}{2}\xx_{3}\xx_{1}\xx_{2} - \frac{1}{2}\xx_{3}\xx_{1}\xx_{3} + \xx_{3}\xx_{2}\xx_{1}  \\ & \hspace{1.5cm}- \frac{1}{2}\xx_{3}\xx_{2}^{2}- \frac{1}{2}\xx_{3}\xx_{2}\xx_{3} - \frac{1}{2}\xx_{3}^{2}\xx_{1}- \frac{1}{2}\xx_{3}^{2}\xx_{2} + \xx_{3}^{3}+\cdots\end{align*}
\end{example}

Several of the places in which symmetric power sums pay a key role do not have analogues with any of the noncommutative power sums.  Their relation to other bases in terms of generating functions is less straightforward;\ see the discussion after the proof of Theorem \ref{thm:hpprod} below.  From an algebraic perspective, one important role of the symmetric power sum basis is as they appear in the definition of the Frobenius character map (as is explained, for example, in \cite{Sagan}), where they encode the characters of representations of the symmetric group in $\sym$.  There is a natural analogue of the Frobenius character map which encodes representations of the 0-Hecke algebra using $\nsym$, as explained in \cite{NSYMIV}.  However, there is no similar character defined on representations of the 0-Hecke algebra, so the map is defined quite differently, and none of the noncommutative power sums appear.  Similarly, the symmetric power sums play an important role in defining or simplifying plethysm in the commuting variables, but are not present in the noncommutative story in the same way as explored in \cite{krob1997noncommutative}.

\noindent Using the standard inner product, Definition \ref{def:phipsi_n} indirectly defines bases of $\textnormal{QSym}$ dual to the two kinds of noncommutative power sum symmetric functions.  These \textbf{quasisymmetric power sums}, $\{\psi_\alpha\}$ and $\{\phi_\alpha\}$, were explored in detail by Ballantine et al.\ in \cite{quasipower}, and satisfy
$$\langle \psi_\alpha,\boldsymbol{\psi}_\beta \rangle = \langle \phi_\alpha,\boldsymbol{\phi}_\beta \rangle = z_\alpha \cdot \mathbbm{1}_{\alpha = \beta}.$$
Here, and elsewhere, for any strong composition $\alpha$, $z_\alpha = z_{\text{sort}(\alpha)}$, the coefficient seen before.

While the images under the remaining involutions are not as straightforward, it is easy to see from the expansions of the noncommutative power sums of the first and second kinds in the noncommutative ribbon basis that we have the following theorem.
\begin{theorem}[\cite{non}, Section 3]\label{thm:inv_on_phi_psi}  For any strong composition $\alpha$,
    \begin{align}\boldsymbol{\omega}(\npsi_\alpha)&=(-1)^{|\alpha|-\ell(\alpha)}\npsi_{\alpha^r};
    \\\boldsymbol{\omega}(\nphi_\alpha)&=(-1)^{|\alpha|-\ell(\alpha)}\nphi_{\alpha^r};
    \\\boldsymbol{\psi}(\nphi_\alpha)&=(-1)^{|\alpha|-\ell(\alpha)}\nphi_\alpha;
    \\\boldsymbol{\rho}(\nphi_\alpha)&=\nphi_{\alpha^r}.\end{align}
\end{theorem}
\noindent These allow us to focus only on change of basis going forward in the noncommutative complete symmetric function basis, since expansion in the elementary symmetric functions will follow from applying the $\boldsymbol{\omega}$ map to each side.

\begin{theorem}[\cite{non}, Section 4.2] For $n$ a postive integer,
    \begin{align*}\npsi_n&=\sum_{\beta\modelsstrong n}(-1)^{1 + \ell(\beta)}\beta_{\ell(\beta)} \nh_\beta. \end{align*}
\end{theorem}
\begin{proof}    \begin{align*}\npsi_n&=\sum\limits_{k=0}^{n-1} (-1)^k \nr_{1^k(n-k)}&&\text{by Definition \ref{def:phipsi_n}}
\\&=\sum\limits_{k=0}^{n-1}  \sum_{\beta \succeq 1^k(n-k)} (-1)^{1 + \ell(\beta)}\boldsymbol{h}_\beta &&\text{by (\ref{thm:hrCOB})}
\\&=\sum_{\beta\modelsstrong n}(-1)^{1 + \ell(\beta)}\nh_\beta \sum_{k=0}^{n-1} \mathbbm{1}_{\beta\succeq(1^k,n-k)}
\\&=\sum_{\beta\modelsstrong n}(-1)^{1 + \ell(\beta)}\nh_\beta \sum_{k=0}^{n-1} \mathbbm{1}_{\beta_{\ell(\beta)}\geq n-k}
\\&=\sum_{\beta\modelsstrong n}(-1)^{1 + \ell(\beta)}\beta_{\ell(\beta)} \nh_\beta.\end{align*}
\end{proof}
\noindent Before giving an analogous result for the $\nphi_n$, we first need the following lemma:
\begin{lemma}\label{lem:helpfulc} Let $n$ and $c$ be nonnegative integers.  Then, $$\sum_{k=0}^n\frac{\binom{n}{k}}{{\binom{n+c}{k+c}}}=\frac{n+c+1}{c+1}.$$
\end{lemma}
\begin{proof}
    By induction on $n$.  The base case $n=0$ is trivially true.   Then \begin{align*}\sum_{k=0}^{n+1}\frac{\binom{n+1}{ k}}{\binom{{n+1+c}}{ {k+c}}}&=1+\sum_{k=0}^{n}\frac{\left(\frac{n+1}{n+1-k}\right)\binom{n}{ k}}{\left(\frac{n+1+c}{n+1-k}\right)\binom{{n+c}}{ {k+c}}}\\&=1+\left(\frac{n+1}{n+1+c}\right)\sum_{k=0}^{n}\frac{\binom{n}{ k}}{\binom{{n+c}}{ {k+c}}}\\&\stackrel{\text{I.H.}}{=}1+\left(\frac{n+1}{n+1+c}\right)\left(\frac{n+1+c}{c+1}\right)\\&=\frac{(n+1)+c+1}{c+1}.\end{align*}
\end{proof}

\begin{theorem}[\cite{non}, Section 4.3]\label{thm:phitoh} If $n$ is a positive integer,
    \begin{align*}\nphi_n&=\sum_{\beta\modelsstrong n}(-1)^{\ell(\beta)+1}\frac{n}{\ell(\beta)}\nh_\beta. \end{align*}
\end{theorem}
\begin{proof}    \begin{align*}\nphi_n&= \sum\limits_{\alpha\modelsstrong n} \frac{(-1)^{\ell(\alpha)-1}}{\binom{n-1}{\ell(\alpha)-1}} \nr_\alpha&&\text{by Definition \ref{def:phipsi_n}} \\&=\sum\limits_{\alpha\modelsstrong n} \frac{(-1)^{\ell(\alpha)-1}}{\binom{n-1}{\ell(\alpha)-1}} \sum_{\beta\succeq \alpha}(-1)^{\ell(\alpha)-\ell(\beta)}\nh_\beta&&\text{by (\ref{thm:hrCOB})} 
\\&=\sum_{\beta\modelsstrong n}\nh_\beta (-1)^{\ell(\beta)+1}\sum_{\alpha\preceq \beta}\frac{1}{\binom{n-1}{ \ell(\alpha)-1}}
\\&=\sum_{\beta\modelsstrong n}\nh_\beta (-1)^{\ell(\beta)+1}\sum_{k=1}^n\frac{1}{\binom{n-1}{ k-1}}\sum_{\alpha\preceq \beta}\mathbbm{1}_{\ell(\alpha)=k}
\\&=\sum_{\beta\modelsstrong n}\nh_\beta (-1)^{\ell(\beta)+1}\sum_{k=\ell(\beta)}^n\frac{\binom{n-\ell(\beta) }{ k-\ell(\beta)}}{\binom{n-1}{ k-1}}&&\\&=\sum_{\beta\modelsstrong n}\nh_\beta (-1)^{\ell(\beta)+1}\sum_{k=0}^{n-\ell(\beta)}\frac{\binom{n-\ell(\beta) }{ k}}{\binom{n-1}{ k+\ell(\beta)-1}}\\&=\sum_{\beta\modelsstrong n} (-1)^{\ell(\beta)+1}\frac{n}{\ell(\beta)}\nh_\beta.&&\text{by Lemma \ref{lem:helpfulc}}\end{align*}
\end{proof}

\begin{theorem}[\cite{non}, Proposition 3.3]\label{thm:hpprod} For $n\geq 1$,
    $$\sum_{i=0}^{n-1}\nh_i \npsi_{n-i}=n\nh_n.$$
\end{theorem}
\begin{proof}
    \begin{align*}
        \sum_{i=0}^{n-1}\nh_i \npsi_{n-i}&=\npsi_n+ \sum_{i=1}^{n-1}\nh_i \npsi_{n-i}\\
        &=\sum_{\beta\modelsstrong n}(-1)^{1 + \ell(\beta)}\beta_{\ell(\beta)} \nh_\beta+ \sum_{i=1}^{n-1}\nh_i \sum_{\beta\modelsstrong n-i}(-1)^{1 + \ell(\beta)}\beta_{\ell(\beta)} \nh_\beta\\
        &=\sum_{\beta\modelsstrong n}(-1)^{1+ \ell(\beta)}\beta_{\ell(\beta)} \nh_\beta+ \sum_{\substack{\beta\modelsstrong n\\\ell(\beta)\ge2}}(-1)^{(\ell(\beta))}\beta_{\ell(\beta)} \nh_\beta\\&=n\nh_n.
    \end{align*} 
\end{proof}
\noindent There are other equally natural choices for an analogue of the power sums.  Gelfand et al.\ chose these based on two generating series  relations on the symmetric functions, whose analogues below are each satisfied by only one of the noncommutative power sum symmetric functions.
\begin{theorem}[\cite{non}, Section 3.1]\label{thm:powersumsseries}Let
$$\boldsymbol{\psi}(t) = \sum_{n \in \mathbb{Z}^+} \frac{\npsi_n}{n}t^{n} \text{ and }\boldsymbol{\phi}(t) = \sum_{n \in \mathbb{Z}^+} \frac{\nphi_n}{n}t^n.$$  Then
\begin{eqnarray}
\label{eq:NSymdH=HPsi}
\frac{d}{dt}(\textbf{H}(t)) = \textbf{H}(t)\frac{d}{dt}\boldsymbol{\psi}(t),
\end{eqnarray}

\noindent and 
\begin{eqnarray}\label{eq:NSymdH=HPhi}
\textbf{H}(t) = \exp(\boldsymbol{\phi}(t)),\end{eqnarray}
or equivalently \begin{eqnarray}\label{eq:philog}
\label{NSymdH=HPhi}\boldsymbol{\phi}(t)=\log\left(1+\sum_{k\geq 1}\nh_kt^k\right).\end{eqnarray}
\end{theorem}
\begin{proof} As explained briefly in \cite{non},
    Equation (\ref{eq:NSymdH=HPsi}) follows immediately from Theorem \ref{thm:hpprod} and equation (\ref{eq:NSymdH=HPhi}) is only slightly less straightforward:
    \begin{align*}
   \log \left(1+\sum_{k\geq 1}\nh_k t^k\right)&=\sum_{n=1}^\infty \frac{(-1)^{n-1}}{n}\left(\sum_{k\geq 1}\nh_k t^k\right)^n\\&=\sum_{n=1}^\infty \frac{(-1)^{n-1}}{n}\sum_{\ell(\alpha)=n}
 \nh_{\alpha_1}\nh_{\alpha_2}\cdots\nh_{\alpha_n}t^{\sum_i{\alpha_i}} \\
 &= \sum_{s=1}^\infty \frac{t^s}{s}\sum_{\beta\modelsstrong s}(-1)^{\ell(\beta)+1}\frac{s}{\ell(\beta)}\nh_\beta\\&=\sum_{s=1}^\infty \frac{t^s}{s}\nphi_s,\end{align*}
 where the last equality follows from Theorem \ref{thm:phitoh}.
 \end{proof}

\noindent Note that while the analogous differential equation in commuting variables, $$\displaystyle{\frac{d}{dt}(H(t)) = H(t)P(t)},$$ defines a unique basis of power sums from the basis of homogeneous complete symmetric functions, there are many equivalent ways to write the same relationship that yield distinct noncommutative analogues.  (To see one additional easy example, reversing the order of the right-hand side of (\ref{eq:NSymdH=HPsi}) would yield a different basis than $\{\npsi_n\}$, which Gelfand et al.\ do not name, as the result is sufficiently similar as to not be interesting.)

\begin{corollary}[\cite{non}, Section 3.1] For $n$ a nonnegative integer,
    $$\chi(\npsi_n)=p_n\text{ and }\chi(\nphi_n)=p_n.$$
\end{corollary}

\begin{theorem}[\cite{non}, Section 4.2]\label{thm:htopsi}  For $n$ a nonnegative integer,
    $$\nh_n=\sum_{\beta \modelsstrong n} \frac{1}{\prod_{i=1}^{\ell(\beta)}\sum_{j=1}^i\beta_j} \npsi_\beta.$$
\end{theorem}
\begin{proof} By strong induction on $n$, with the base case of $n=0$ being trivial:
\begin{align*}\nh_n&=\frac{1}{n}\sum_{i=0}^{n-1}\nh_i\npsi_{n-i}\\
&=\frac{1}{n}\sum_{i=0}^{n-1}\sum_{\beta\modelsstrong i}\frac{1}{\prod_{k=1}^{\ell(\beta)}\sum_{j=1}^k\beta_j} \npsi_\beta\npsi_{n-i}\\
&=\sum_{\gamma \modelsstrong n} \frac{1}{\prod_{i=1}^{\ell(\gamma)}\sum_{j=1}^i\gamma_j} \npsi_\gamma,
\end{align*}
where $\gamma=(\beta,n-i).$  
\end{proof}

\begin{theorem}[\cite{non}, Section 4.3]\label{thm:htophi}  For $n$ a nonnegative integer,
    $$\nh_n=\sum_{\beta \modelsstrong n} \frac{1}{\ell(\beta)!\prod_{i=1}^{\ell(\beta)}\beta_i} \nphi_\beta.$$
\end{theorem}
\begin{proof} By induction on $n$, with the base case of $n=0$ being trivial.  We simultaneously prove both the statement, and its image under the 
$\boldsymbol{\omega}$ map: $$\nel_n=\sum_{\beta \modelsstrong n} \frac{1}{\ell(\beta)!\prod_{i=1}^{\ell(\beta)}\beta_i} (-1)^{|\beta|-\ell(\beta)}\nphi_\beta.$$

\noindent Then 
\begin{align*}\nh_n&=\sum_{i=0}^{n-1}(-1)^{n-i-1}\nh_i\nel_{n-i}&&\text{by (\ref{NSymeandh})}\\
&=\sum_{i=0}^{n-1}(-1)^{n-i-1}\sum_{\beta \modelsstrong i} \left(\frac{1}{\ell(\beta)!\prod_{i=1}^{\ell(\beta)}\beta_i} \nphi_\beta\right)\\&\hspace{1.6in}\left(\sum_{\gamma \modelsstrong n-i} \frac{1}{\ell(\gamma)!\prod_{i=1}^{\ell(\gamma)}\gamma_i} (-1)^{|\gamma|-\ell(\gamma)}\nphi_\gamma\right)&&\text{by I.H.}\\
&=\sum_{\delta\modelsstrong n}\nphi_\delta\frac{ 1}
{\ell(\delta)!\prod_{i=1}^{\ell(\delta)}\delta_i}\sum_{j=1}^{\ell(\delta)}(-1)^{j-1}\binom{\ell(\delta)}{ j}&& \delta=(\beta,\gamma)\text{ and }j=\ell(\beta)\\
&=\sum_{\delta \modelsstrong n} \frac{1}{\ell(\delta)!\prod_{i=1}^{\ell(\delta)}\delta_i} \nphi_\delta.&&\text{by Binomial Theorem}
\end{align*}
  Applying the $\boldsymbol{\omega}$ map again completes the induction.
\end{proof}
\begin{corollary}[Gelfand et al.\ \cite{non}]
    Both $\{\boldsymbol{\psi}_{\alpha}\}_{ \alpha \modelsstrong n}$ and $\{\boldsymbol{\phi}_\alpha\}_{ \alpha \modelsstrong n}$ are bases of $\textbf{NSym}^n$. 
\end{corollary} 
\noindent A number of the reoccurring statistics above are given names in \cite{non} and repeated here.  They need to be generalized to cover extending the above results to the multiplicative bases.

\begin{definition} Let $\beta = (\beta_1,\beta_2,\dots,\beta_{\ell(\beta)})$ be a strong composition.  Denote the \textbf{last part} of $\beta$ by 
$$\operatorname{lp}(\beta) = \beta_{\ell(\beta)}.$$
Also say $\beta$'s \textbf{product of partial sums},  \textbf{product}, and \textbf{special product}, respectively, are
\begin{eqnarray}
\pi_u(\beta) &=& \prod_{i=1}^{\ell(\beta)} \sum_{k=1}^i \beta_k = \beta_1(\beta_1 + \beta_2)\cdots(\beta_1 + \beta_2 + \cdots + \beta_{\ell(\beta)}); \label{piu}
\\ \nonumber
\\\prod \beta &=& \beta_1\beta_2\cdots \beta_{\ell(\beta)}; \label{prod}
\\ \nonumber
\\ \operatorname{sp}(\beta) &=& \ell(\beta)!\prod \beta. \label{sp}
\end{eqnarray}
Recall that if $\beta \preceq \alpha$, $\beta^{(i)}$ is the subcomposition of $\beta$ which sums to $\alpha_i$ for $i =1,\dots,\ell(\alpha)$.  Extend the above definitions to refinements $\beta \preceq \alpha$, with $|\beta^{(i)}| = \alpha_i$, $i \in [\ell(\alpha)]$:

\begin{minipage}[adjusting]{7cm}
\begin{eqnarray}
\operatorname{lp}(\beta,\alpha) &=& \prod_{i=1}^{\ell(\alpha)} \operatorname{lp}(\beta^{(i)}) \label{lpr}
\\ \ell(\beta,\alpha) &=& \prod_{i=1}^{\ell(\alpha)} \ell(\beta^{(i)}) \label{lr}
\end{eqnarray}
\end{minipage} 
\begin{minipage}[adjusting]{7.2cm}
\begin{eqnarray}
\pi_u(\beta,\alpha) &=& \prod_{i=1}^{\ell(\alpha)} \pi_u(\beta^{(i)}) \label{piur}
\\\operatorname{sp}(\beta,\alpha) &=& \prod_{i=1}^{\ell(\alpha)} \operatorname{sp}(\beta^{(i)}) \label{spr}
\end{eqnarray}
\end{minipage}
\end{definition}

\noindent In \cite{non}, all of the following change-of-basis equations were established.
\begin{theorem}\label{thm:nsymCOBeqns} For $\alpha\modelsstrong n$,
\begin{align}
&\boldsymbol{h}_\alpha = \sum_{\beta \preceq \alpha} (-1)^{|\alpha| - \ell(\beta)}\boldsymbol{e}_\beta; && \boldsymbol{e}_\alpha = \sum_{\beta \preceq \alpha} (-1)^{|\alpha| - \ell(\beta)}\boldsymbol{h}_\beta\label{ehCOB}
\\[1em]&\boldsymbol{h}_\alpha = \sum_{\beta \succeq \alpha} \boldsymbol{r}_\beta  &&\boldsymbol{r}_\alpha = \sum_{\beta \succeq \alpha} (-1)^{\ell(\alpha) - \ell(\beta)}\boldsymbol{h}_\beta \label{hrCOB}
\\[1em]&\boldsymbol{e}_\alpha = \sum_{\beta^t \succeq \alpha^r} \boldsymbol{r}_\beta  &&\boldsymbol{r}_\alpha = \sum_{\beta^r \succeq \alpha^t} (-1)^{\ell(\alpha^t) - \ell(\beta)}\boldsymbol{e}_\beta \label{erCOB}
\\[1em]&\boldsymbol{h}_\alpha = \sum_{\beta \preceq \alpha} \frac{1}{\pi_u(\beta,\alpha)}\boldsymbol{\psi}_\beta &&\boldsymbol{\psi}_\alpha = \sum_{\beta \preceq \alpha} (-1)^{\ell(\beta) - \ell(\alpha)}\operatorname{lp}(\beta,\alpha) \boldsymbol{h}_\beta \label{hPsiCOB}
\\[1em]&\boldsymbol{h}_\alpha = \sum_{\beta \preceq \alpha} \frac{1}{\operatorname{sp}(\beta,\alpha)}\boldsymbol{\phi}_\beta 
&&\boldsymbol{\phi}_\alpha = \sum_{\beta \preceq \alpha} (-1)^{\ell(\beta) - \ell(\alpha)} \frac{\prod \alpha}{\ell(\beta,\alpha)}\boldsymbol{h}_\beta \label{hPhiCOB}
\\[1em]&\boldsymbol{e}_\alpha = \sum_{\beta \preceq \alpha}  \frac{(-1)^{|\alpha| - \ell(\beta)}}{\pi_u(\beta^r,\alpha^r)}\boldsymbol{\psi}_\beta  &&\boldsymbol{\psi}_\alpha = \sum_{\beta \preceq \alpha} (-1)^{|\alpha| - \ell(\beta)}\operatorname{lp}(\beta^r,\alpha^r)\boldsymbol{e}_\beta \label{ePsiCOB}
\\[1em]&\boldsymbol{e}_\alpha = \sum_{\beta \preceq \alpha}  \frac{(-1)^{|\alpha| - \ell(\beta)}}{\operatorname{sp}(\beta,\alpha)}\boldsymbol{\phi}_\beta
&&\boldsymbol{\phi}_\alpha = \sum_{\beta \preceq \alpha} (-1)^{|\alpha| - \ell(\beta)} \frac{\prod \alpha}{\ell(\beta,\alpha)}\boldsymbol{e}_\beta \label{ePhiCOB}
\end{align}
\end{theorem}

\begin{proof} Line (\ref{ehCOB}) comes from Theorem \ref{thm:heCOB} and an application of the $\boldsymbol{\psi}$ map (see line (\ref{NSyminvos})).  Lines (\ref{hrCOB}) and (\ref{erCOB}) were shown in Theorem \ref{thm:ehtor}.  The remaining lines follow from Theorems \ref{thm:phitoh}, \ref{thm:htopsi}, and \ref{thm:htophi}, the multiplicative definitions of $\nh_\alpha$, $\nel_\alpha$, $\nphi_\alpha$, and $\npsi_\alpha$, and applications of the $\boldsymbol{\omega}$ map (see Theorem \ref{thm:inv_on_phi_psi}).
\end{proof}

\noindent There are two remaining explicit change-of-basis formulas in \cite{non}, both of which follow from the product formula for $\nr_\alpha$, Theorem \ref{thm:ribbon_mult} above.
We need the following definition:
\begin{definition}
    Let $\alpha$ and $\beta$ be strong compositions of $n$.  Let $\gamma=\set^{-1}(\set(\alpha)\cup \set(\beta))$.  Then $\gamma \preceq \beta$, so we can let $\gamma=(\gamma^{(1)},\gamma^{(2)},\dots, \gamma^{(k)})$ give the subsequences such that $|\gamma^{(j)}| = \beta_j$.
    Then the \textbf{ribbon decomposition of $\alpha$ with respect to $\beta$} is $$\operatorname{rd}(\alpha,\beta)=(\gamma^{(1)},\gamma^{(2)},\cdots, \gamma^{(k)}).$$
    Furthermore, let $$\operatorname{psr}(\alpha,\beta)= (-1)^{|\gamma^{(1)}|-1}(-1)^{|\gamma^{(2)}|-1}\cdots (-1)^{|\gamma^{(k)}|-1}\mathbbm{1}_{\gamma^{(1)}\text{ is a hook }}\mathbbm{1}_{\gamma^{(2)}\text{ is a hook }}\cdots \mathbbm{1}_{\gamma^{(k)}\text{ is a hook }}$$
    and $$\operatorname{phr}(\alpha,\beta)= \frac{(-1)^{\ell(\gamma^{(1)})-1}}{\binom{n-1}{\ell(\gamma^{(1)})-1}}\cdots \frac{(-1)^{\ell(\gamma^{(k)})-1}}{\binom{n-1}{\ell(\gamma^{(k)})-1}}.$$
\end{definition}
\begin{example}  Let $\alpha=(1,3,2,4,4)$ and $\beta=(4,3,5,2)$.  Then $\gamma=(1,3,2,1,3,2,2)$ and we have a ribbon decomposition of $\alpha$ with respect to $\beta$ is $$\operatorname{rd}(\alpha,\beta)=((1,3),(2,1),(3,2),(2))$$
\end{example}
\noindent As first observed by Gelfand et al., it is easy to see from Definition \ref{def:phipsi_n} that
\begin{theorem}[\cite{non}, Prop.~4.23 and Prop.~4.27] Let $\alpha$ be a strong composition.  Then
$$\npsi_\alpha=\sum_{\beta\modelsstrong n}\operatorname{psr}(\beta,\alpha)\nr_\alpha\text{ and }\nphi_\alpha=\sum_{\beta\modelsstrong n}\operatorname{phr}(\beta,\alpha)\nr_\alpha.$$
\end{theorem}
\begin{remark*}  Omitted here is the work in \cite{non} towards change of basis between $\nphi$ and $\npsi$, which did not result in as nice of change-of-basis formulas. The interested reader should consult Section 4.10 in \cite{non}, where there is a somewhat more complex formula for $\nphi_n$ in terms of $\{\npsi_{\alpha}\}_{\alpha\modelsstrong n}$.
\end{remark*}

\begin{note*} In \cite{intro}, the authors mention that $\{\omega(M_\alpha)\}$ yields the dual basis to $\{\boldsymbol{e}_\alpha\}$.  For our purposes later, it will be most natural to utilize $\psi$ to define a quasisymmetric analogue to the forgotten symmetric functions.  While the result is not quite a dual to $\{\boldsymbol{e}_\alpha\}$, the resulting basis still restricts to the forgotten symmetric functions under the forgetful map.  Therefore, define the \textbf{forgotten quasisymmetric function} (associated to $\alpha$) to be 
$$\For_{\alpha} = \psi(M_\alpha).$$
\end{note*}

\noindent Then, by duality, the following corresponding equations can be established.
\begin{corollary}\label{cor:qsymCOBeqns} For $\alpha\modelsstrong n$,
\begin{align}
&F_\alpha = \sum_{\beta \preceq \alpha} M_\beta &&M_\alpha = \sum_{\beta \preceq \alpha} (-1)^{\ell(\beta)-\ell(\alpha)} F_\beta \label{MFCOB}
\\[1em]&F_\alpha = \sum_{\beta^r \preceq \alpha^t} \For_\beta &&\For_\alpha = \sum_{\beta^t \preceq \alpha^r} (-1)^{\ell(\beta^t) - \ell(\alpha)}F_\beta \label{FForCOB}
\\[1em]&
\For_\alpha = \sum_{\beta \succeq \alpha} (-1)^{\ell(\alpha) - |\beta|}M_\beta && M_\alpha = \sum_{\beta \succeq \alpha} (-1)^{\ell(\alpha) - |\beta|}\For_\beta.\label{MForCOB}
\\[1em]&\psi_\alpha = z_\alpha \sum_{\beta \succeq \alpha} \frac{1}{\pi_u(\alpha,\beta)}M_\beta &&M_\alpha = \sum_{\beta \succeq \alpha} (-1)^{\ell(\alpha) - \ell(\beta)}\operatorname{lp}(\alpha,\beta) \frac{\psi_\beta}{z_\beta} \label{MPsiCOB}
\\[1em]&\phi_\alpha = z_\alpha\sum_{\beta \succeq \alpha} \frac{1}{sp(\alpha,\beta)}M_\beta &&M_\alpha = \sum_{\beta \succeq \alpha} (-1)^{\ell(\alpha)-  \ell(\beta)}\frac{\prod \beta}{\ell(\alpha,\beta)}\frac{\phi_{\beta}}{z_\beta} \label{MPhiCOB}
\\[1em]&\psi_\alpha = z_\alpha \sum_{\beta \succeq \alpha}  \frac{(-1)^{|\beta| - \ell(\alpha)}}{\pi_u(\alpha^r,\beta^r)}\For_\beta &&\For_\alpha = \sum_{\beta \succeq \alpha} (-1)^{|\beta| - \ell(\alpha)} \label{ForPsiCOB} \operatorname{lp}(\alpha^r,\beta^r)\frac{\psi_\beta}{z_\beta}
\\[1em]&\phi_\alpha = z_\alpha \sum_{\beta \succeq \alpha}  \frac{(-1)^{|\beta| - \ell(\alpha)}}{sp(\alpha,\beta)}\For_\beta &&\For_\alpha = \sum_{\beta \succeq \alpha} (-1)^{|\beta| - \ell(\alpha)} \frac{\prod \beta}{\ell(\alpha,\beta)}\frac{\phi_\beta}{z_\beta} \label{ForPhiCOB}
\end{align}
  \end{corollary}
  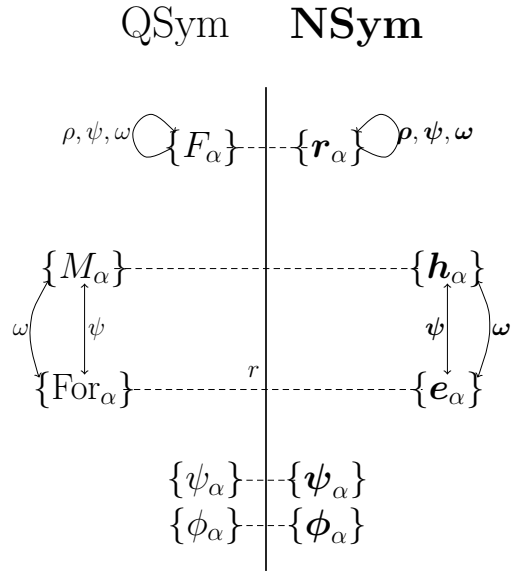
\begin{figure}[ht]
    $$\begin{tikzpicture}
    \useasboundingbox (-3.5,-3) rectangle (3.5,4.5);
      \scope[transform canvas={scale=.8}]
\node at (-1.05,3) {\Large{$\{F_\alpha\}$}};
\node at (1.05,3) {\Large{$\{\boldsymbol{r}_\alpha\}$}};
\node at (3,1) {\Large{$\{\boldsymbol{h}_\alpha\}$}};
\node at (-3,1) {\Large{$\{M_\alpha\}$}};
\node at (1,-2.5) {\Large{$\{\boldsymbol{\psi}_\alpha\}$}};
\node at (1,-3.25) {\Large{$\{\boldsymbol{\phi}_\alpha\}$}};
\node at (-1,-2.5) {\Large{$\{\psi_\alpha\}$}};
\node at (-1,-3.25) {\Large{$\{\phi_\alpha\}$}};
\node at (-3,-1) {\Large{$\{\For_\alpha\}$}};
\node at (-.2,-.7) {$r$};
\node at (3,-1) {\Large{$\{\boldsymbol{e}_\alpha\}$}};
\node at (-2.8,0) {$\psi$};
\node at (2.8,0) {$\boldsymbol{\psi}$};
\node at (-4.05,0) {$\omega$};
\node at (3.9,0) {$\boldsymbol{\omega}$};
\node at (-2.8,3.25) {$\rho,\psi,\omega$};
\node at (2.8,3.25) {$\boldsymbol{\rho},\boldsymbol{\psi},\boldsymbol{\omega}$};
\node at (-1.5,5) {\LARGE{QSym}};
\node at (1.5,5) {\LARGE{\textbf{NSym}}};

\draw [thick] (0,4)--(0,-4);
\draw [densely dashed] (-2.51,1)--(2.55,1);
\draw [densely dashed] (-2.36,-1)--(2.59,-1);
\draw [densely dashed](-0.65,3)--(0.7,3);
\draw [densely dashed](-0.5,-3.25)--(0.45,-3.25);
\draw [densely dashed](-0.5,-2.5)--(0.5,-2.5);

\draw [-latex][->] plot[smooth, tension=.7] coordinates {(-1.55,3) (-2,2.85)(-2.2,3.25)(-1.9,3.55)(-1.5,3.25)};
\draw [-latex][->] plot[smooth, tension=.7] coordinates {(1.55,3) (2,2.85)(2.2,3.25)(1.9,3.55)(1.5,3.25)};

\draw [-latex][<->] plot[smooth, tension=.7] coordinates {(-3,0.75)(-3,-0.75)};
\draw [-latex][<->] plot[smooth, tension=.7] coordinates {(3,0.75)(3,-0.75)};
\draw [-latex][<->] plot[smooth, tension=.7] coordinates {(-3.6,0.8)(-3.9,0)(-3.75,-0.8)};
\draw [-latex][<->] plot[smooth, tension=.7] coordinates {(3.5,0.8)(3.7,0)(3.5,-0.8)};
\endscope
\end{tikzpicture}$$
    \caption{\hspace{0.1cm}Automorphisms and Duality of $\textnormal{QSym}$ and $\textbf{NSym}$}
    \label{QNDuality}
\end{figure}
\noindent Figure \ref{QNDuality}, on the next page, gives a diagram analogous to Figure \ref{Duality of Sym} depicting the results from this section.  Horizontal dashed segments indicate duality once again, unless marked by an `r' to denote reversing the order.  Note the vertical edges labeled $\psi$ and $\boldsymbol{\psi}$ indicate that a basis element indexed by $\alpha$ is sent to its counterpart (also indexed by $\alpha$) in the other basis.  All other edges, outside set braces, indicate basis elements are most often not sent to their exact counterparts in the image set.  Suppressed are four loops that would be labeled with either $\rho$ or $\boldsymbol{\rho}$, as well as the combinatorial coefficients $z_\alpha$ at the bottom, for readability.

We end by discussing two ways in which a number of the change-of-basis results in $\sym$ have been condensed, and show analogous results in the dual bases of $\qsym$ and $\nsym$.

\newpage

\section{Relations between transition matrices}
\noindent A well-known diagram in Macdonald \cite{Macdonald}, p.105, reproduced in Figure \ref{fig:symdiagram} (on the next page), summarizes the relationship between various transitions matrices in $\sym$.  Our next goal is to reproduce a similar diagram for $\qsym$ and $\nsym$.

\begin{notation}
Let $b = \{b_\beta\}_{\beta \in B}$ and $a = \{a_\alpha\}_{\alpha \in A}$ be bases of some vector space $V$ with ordered indexing sets $B$ and $A$, respectively.  Then the \textbf{change-of-basis matrix from $\{b_\beta\}_{\beta \in B}$ to $\{a_\alpha\}_{\alpha \in A}$} is the matrix $\text{M}(a,b)$ for which 
$$[a_\alpha]_{\alpha \in A} = \text{M}(a,b)[b_\beta]_{\beta \in B}.$$
Here, $[a_\alpha]_{\alpha \in A}$ and $[b_\beta]_{\beta \in B}$ denote the column vectors of all the basis elements $a_\alpha$ and $b_\beta$, respectively.
\end{notation}
\noindent Let $n$ be a positive integer and let $p(n)$ be the number of integer partitions of $n$.  Then, in the context of $\sym^n$, Macdonald defines the following $p(n) \times p(n)$ matrices:
\begin{itemize}
    \item $K=M(s,m)=[K_{\lambda,\mu}]_{\lambda,\mu}$, where $K_{\lambda, \mu}$ is the \textbf{Kostka number} $K_{\lambda,\mu}$, the number of semi-standard Young tableaux of shape $\lambda \vdash n$ and type $\mu$;
    \item $J = \left(\mathbbm{1}_{\lambda = \mu^t} \right)_{\mu,\lambda}$;
    \item $z = (z_\lambda \cdot \mathbbm{1}_{\lambda = \mu})_{\mu,\lambda}$;
    \item $L =L(p,m)=[L_{\lambda,\mu}]_{\lambda,\mu}$, where $L_{\lambda,\mu} = \left|\left\{f:[\ell(\lambda)] \rightarrow \mathbb{Z}^+ ~\Big|~ \mu = \Big(\sum_{f(j) = i} \lambda_j\Big)_{i=1}^\infty \right\}\right|_.$
    (See the language of (6.9) in \cite{Macdonald}.)
\end{itemize}

\noindent 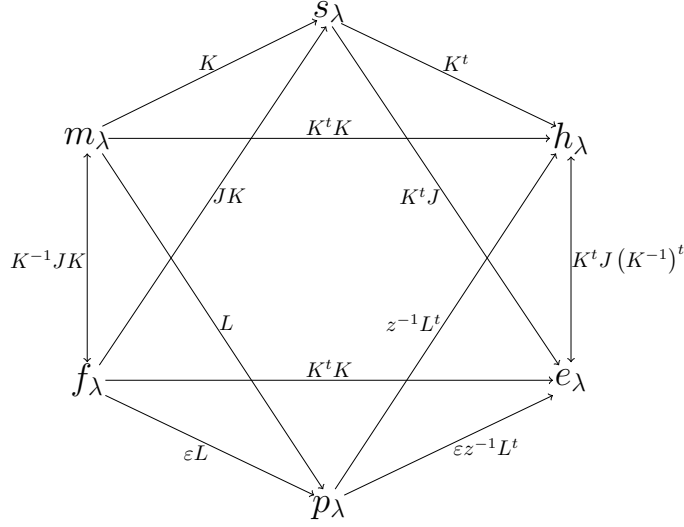
\begin{figure}
$$\begin{tikzpicture}
    \useasboundingbox (-4,-3.5) rectangle (4.5,3.5);
      \scope[transform canvas={scale=.8}]
\node at (0,4.1) {\Large{$s_\lambda$}};
\node at (-4,2) {\Large{$m_\lambda$}};
\node at (-4,-2) {\Large{$f_\lambda$}};
\node at (4,2) {\Large{$h_\lambda$}};
\node at (4,-2) {\Large{$e_\lambda$}};
\node at (0,-4.1) {\Large{$p_\lambda$}};

\node at (-4.65,0) {\footnotesize{$K^{-1}JK$}};
\node at (4.95,0) {\footnotesize{$K^tJ\left(K^{-1}\right)^t$}};
\node at (-2,3.25) {\footnotesize{$K$}};
\node at (2.1,3.25) {\footnotesize{$K^t$}};

\draw [-latex][->] plot[smooth, tension=.7] coordinates {(0.2,3.9) (3.75,2.2)};
\draw [-latex][<-] plot[smooth, tension=.7] coordinates {(-0.2,3.95) (-3.75,2.2)};

\draw [-latex][->] plot[smooth, tension=.7] coordinates {(-3.65,2) (3.65,2)};
\node at (0,2.2) {\footnotesize{$K^tK$}};

\draw [-latex][<-] plot[smooth, tension=.7] coordinates {(-0.05,3.85) (-3.8,-1.75)};
\node at (-1.65,1.05) {\footnotesize{$JK$}};
\draw [-latex][->] plot[smooth, tension=.7] coordinates {(0.05,3.85) (3.8,-1.75)};
\node at (1.5,1.05) {\footnotesize{$K^tJ$}};

\draw [-latex][->] plot[smooth, tension=.7] coordinates {(-3.7,-2) (3.7,-2)};
\node at (0,-1.8) {\footnotesize{$K^tK$}};

\draw [-latex][->] plot[smooth, tension=.7] coordinates {(0.1,-3.8) (3.75,1.75)};
\node at (1.4,-1.05) {\footnotesize{$z^{-1}L^t$}};
\draw [-latex][<-] plot[smooth, tension=.7] coordinates {(-0.1,-3.8) (-3.75,1.75)};
\node at (-1.7,-1.05) {\footnotesize{$L$}};

\draw [-latex][->] plot[smooth, tension=.7] coordinates {(0.25,-3.9) (3.7,-2.25)};
\node at (2.6,-3.1) {\footnotesize{$\varepsilon z^{-1} L^t$}};
\draw [-latex][<-] plot[smooth, tension=.7] coordinates {(-0.25,-3.9) (-3.7,-2.25)};
\node at (-2.2,-3.2) {\footnotesize{$\varepsilon L$}};

\draw [-latex][<->] plot[smooth, tension=.7] coordinates {(-4,1.75) (-4,-1.7)};
\draw [-latex][<->] plot[smooth, tension=.7] coordinates {(4,1.7) (4,-1.7)};
\endscope
\end{tikzpicture} 
$$
\caption{~Matrix Change-of-Basis Expressions in $\sym$}\label{fig:symdiagram}
\end{figure}

\subsection{Change-of-Basis Expressions in \texorpdfstring{$\nsym$}{NSym} and \texorpdfstring{$\qsym$}{QSym}}
In this subsection, we generalize the results on matrix equations in \cite{Macdonald} to the spaces of $\textnormal{QSym}$ and $\textbf{NSym}$.  The proofs in this section are quite similar to the originals in \cite{Macdonald}, with only minor additional complexity from the multiple isometries and power sum bases in $\qsym$ and $\nsym$.

\noindent It is helpful to recall three well-known facts about general change of basis, and two additional facts  pertaining to our involutions:  
\begin{enumerate}[label=(\Roman*)]
    \item For any bases $\{a_\alpha\}$, $\{b_\alpha\}$, and $\{d_{\alpha}\}$ in a fixed vector space, $\text{M}(d,b)\text{M}(b,a) = \text{M}(d,a)$;
    \item For any bases $\{a_\alpha\}$, and $\{b_\alpha\}$, $\text{M}(b,a) = \text{M}(a,b)^{-1}$.
\item For any bases $\{A_\alpha\}, \{B_\alpha\} \subset \textnormal{QSym}$, and their respective dual bases $\{\boldsymbol{a}_\alpha\}, \{\boldsymbol{b}_\alpha\} \subset \textbf{NSym}$, $\text{M}(A,B) = \text{M}(\boldsymbol{b},\boldsymbol{a})^t.$
 \item For any bases $\{A_\alpha\}, \{B_\alpha\} \subset \textnormal{QSym}$, $\text{M}(A,B) = \text{M}(\psi(A),\psi(B)) = \text{M}(\rho(A),\rho(B)) = \text{M}(\omega(A),\omega(B));$ 
    \item  For any bases $\{\boldsymbol{a}_\alpha\}, \{\boldsymbol{b}_\alpha\} \subset \textbf{NSym}$$, \text{M}(\boldsymbol{a},\boldsymbol{b}) = \text{M}(\boldsymbol{\psi}(\boldsymbol{a}),\boldsymbol{\psi}(\boldsymbol{b})) = \text{M}(\boldsymbol{\rho}(\boldsymbol{a}),\boldsymbol{\rho}(\boldsymbol{b})) = \text{M}(\boldsymbol{\omega}(\boldsymbol{a}),\boldsymbol{\omega}(\boldsymbol{b})).$
\end{enumerate}

\noindent With these at our disposal, we imitate the approach from Chapter 6 of \cite{Macdonald}.

\noindent In the classical case, the matrix of Kostka numbers is defined, $K = \text{M}(s,m)$.  Now let $n$ be a positive integer and define the quasisymmetric analogue to $K$ indexed by the $2^{n-1} \times 2^{n-1}$ strong compositions of $n$ (see (\ref{MFCOB})),
\begin{align}\mathcal{K} = \text{M}(F,M) = \left(\mathbbm{1}_{\beta \preceq \alpha}\right)_{\alpha,\beta},\label{FM}\end{align}
where $\{F_\alpha\}$ and $\{M_\alpha\}$ are the Gessel Fundamental and monomial quasisymmetric function bases, respectively.  Reusing the same labels, similarly define the analogous matrices
\begin{itemize}
    \item $\varepsilon = (\varepsilon(\alpha) \cdot \mathbbm{1}_{\alpha = \beta})_{\alpha,\beta} =\left((-1)^{|\alpha|-\ell(\alpha)} \cdot \mathbbm{1}_{\alpha = \beta} \right)_{\alpha,\beta}$ and
    \item $z = (z_\alpha \cdot \mathbbm{1}_{\alpha = \beta})_{\alpha,\beta}$.
\end{itemize}

\noindent Then, by (II) and (III), and from lines (\ref{hrCOB}) and (\ref{MFCOB}), we immediately have the following change-of-basis matrices as well.
\begin{align}\mathcal{K}^{-1} &= \text{M}(M,F) = \big((-1)^{\ell(\beta) - \ell(\alpha)} \cdot \mathbbm{1}_{\beta \preceq \alpha}\big)_{\alpha,\beta}=\varepsilon \mathcal{K}\varepsilon\label{MF}\\
(\mathcal{K}^t)^{-1}& = \text{M}(\boldsymbol{r},\boldsymbol{h}) = \big((-1)^{\ell(\alpha) - \ell(\beta)}\cdot \mathbbm{1}_{\alpha \preceq \beta}\big)_{\alpha,\beta}=\varepsilon \mathcal{K}^t\varepsilon\label{rh}\\
\mathcal{K}^t& = \text{M}(\boldsymbol{h},\boldsymbol{r}) = (\mathbbm{1}_{\alpha \succeq \beta})_{\alpha,\beta}
\label{hr}\end{align}
Here, $\{\boldsymbol{r}_\alpha\}$ and $\{\boldsymbol{h}_\alpha\}$ are the noncommutative ribbon and noncommutative complete homogeneous symmetric function bases of $\textbf{NSym}$, respectively.  Note that in contrast to $\sym$, we do not need to use a separate matrix $\mathcal{K}^{-1}$ (and note that $\varepsilon^{-1} = \varepsilon$).

\begin{remark*} There are several change-of-basis matrices in $\sym$ that do not make sense with the defined bases for $\nsym$ and $\qsym$. 
 For example $\text{M}(h,m)$ is well defined, but  $\{\nh_\alpha\}_{\alpha\modelsstrong n}$ and $\{M_\alpha\}_{\alpha\modelsstrong n}$ are in different spaces and thus there is no analogous matrix here.  In principle, one could define a new basis based on these change-of-basis results; for example, we could define $\{\boldsymbol{m}_\alpha\}_{\alpha \modelsstrong n} \subset \nsym$ so that $\text{M}(\boldsymbol{m},\nh)=\mathcal{K}^t\mathcal{K}$.  The result is consistent (so we can define $\boldsymbol{m}$ by its relation to any of the other bases and get the same basis), but a bit algebraically uninteresting. Because of duality, this definition would yield $$\boldsymbol{m}_\alpha= \sum_{\beta \preceq \alpha} (-1)^{\ell(\beta)-\ell(\alpha)} \nr_\beta,$$ which is not far from $$M_\alpha = \sum_{\beta \preceq \alpha} (-1)^{\ell(\beta)-\ell(\alpha)} F_\beta.$$  In some cases, where the hypothetical resulting basis is either $\nr$- (or $F$)-positive, it would encode some representation via one of the two Frobenius maps of Krob and Thibon in \cite{NSYMIV}, which encode modules of the type A $0$-Hecke Algebra.  See \cite{McCloskey} for additional details.   
\end{remark*}

\begin{remark*} Once again, our definition of $\boldsymbol{m}_\alpha$ does not depend on which basis we ``start from'' in $\sym$, as indicated by the three equalities above.
\end{remark*}

\noindent To continue, we would like to imitate the usage of the matrix $J$ from Chapter 6 of \cite{Macdonald}.  With three analogs to the involution $\omega$ to choose from, there are three natural generalizations in each of $\qsym$ and $\nsym$.

\begin{definition}
Let $J_f = (J_f)_{\alpha,\beta} = (\mathbbm{1}_{f(F_\alpha) = F_\beta})_{\alpha,\beta}$ for $f = \psi, \rho$, or $\omega$.
\end{definition}

\begin{remark*} Equivalently, $J_f = (J_f)_{\alpha,\beta} = (\mathbbm{1}_{f(\boldsymbol{r}_\alpha) = \boldsymbol{r}_\beta})_{\alpha,\beta}$ for any of $f = \boldsymbol{\psi}, \boldsymbol{\rho},$ or $\boldsymbol{\omega}$.  Furthermore, since all of the maps $f$ under consideration are involutions, it follows that $J_f^2$ is the identity matrix, and thus $J_f^{-1} = J_f$, just as we saw in the classical case.  Since $J_f$ is a permutation matrix, it must also be orthogonal, and thus its own transpose, $J_f^t = J_f$.
\end{remark*}

\noindent As we alluded to when defining the forgotten quasisymmetric functions, it turns out most useful to us will be $J_\psi$, corresponding to the involutions $\psi: \textnormal{QSym} \rightarrow \textnormal{QSym}$ and $\boldsymbol{\psi}:\textbf{NSym} \rightarrow \textbf{NSym}$, which respectively preserve the indexing on the relevant bases (see Figure 3.1).  In fact, if we choose to order the integer compositions which index these matrices $J$ with any ordering consistent with reverse lexicographic ordering, the matrix $J_\psi$ has the particularly simple form with $1$s on the anti diagonal and $0$s elsewhere.

Our first result shows that each of the remaining permissible relations from the table on page 101 of \cite{Macdonald} generalize appropriately to the spaces of $\textnormal{QSym}$ (or $\textbf{NSym}$), with $J_\psi$ in place of $J$.

\begin{theorem} \label{genMac} We have the following change-of-basis relations.
\begin{enumerate}[label=(\roman*)]
\item $\text{M}(M,\For) = \text{M}(\For,M) = \varepsilon\mathcal{K}\varepsilon J_\psi \mathcal{K};$
\item $\text{M}(\For,F) =\varepsilon\mathcal{K}\varepsilon J_\psi;$
\item $\text{M}(F,\For) = J_\psi\mathcal{K};$
\item $\text{M}(\boldsymbol{h},\boldsymbol{e}) = \text{M}(\boldsymbol{e},\boldsymbol{h}) = \mathcal{K}^tJ_{{\psi}}\varepsilon\mathcal{K}^t\varepsilon;$
\item $\text{M}(\boldsymbol{e},\boldsymbol{r}) = \mathcal{K}^tJ_{{\psi}};$
\item $\text{M}(\boldsymbol{r},\boldsymbol{e}) = J_{{\psi}}\varepsilon\mathcal{K}^t\varepsilon.$
\end{enumerate}
\end{theorem}

\begin{proof} Utilizing (\ref{rh}), (I), and (V), we can immediately establish \textit{(vi)}:
$$\text{M}(\boldsymbol{r},\boldsymbol{e}) = \text{M}(\boldsymbol{\psi} \boldsymbol{r},\boldsymbol{\psi} \boldsymbol{e}) = \text{M}(\boldsymbol{\psi} \boldsymbol{r}, \boldsymbol{h}) = \text{M}(\boldsymbol{\psi} \boldsymbol{r},\boldsymbol{r})\text{M}(\boldsymbol{r},\boldsymbol{h}) = J_\psi\varepsilon\mathcal{K}^t\varepsilon.$$
Then (II) gives \textit{(v)}:
$$\text{M}(\boldsymbol{e},\boldsymbol{r}) = \text{M}(\boldsymbol{r},\boldsymbol{e})^{-1} =  \left(J_\psi (\mathcal{K}^t)^{-1}\right)^{-1} = \mathcal{K}^tJ_{{\psi}}.$$

\noindent With (\ref{rh}), (\ref{hr}), \textit{(vi)}, and \textit{(v)}, we may use (I) again to establish \textit{(iv)}:
\begin{eqnarray*}
\text{M}(\boldsymbol{e},\boldsymbol{h}) &=& \text{M}(\boldsymbol{e},\boldsymbol{r})\text{M}(\boldsymbol{r},\boldsymbol{h})
\\ &=& \mathcal{K}^tJ_{{\psi}} (\mathcal{K}^t)^{-1}
\\ &=& \text{M}(\boldsymbol{h},\boldsymbol{r})\text{M}(\boldsymbol{r},\boldsymbol{e})
\\ &=& \text{M}(\boldsymbol{h},\boldsymbol{e}).
\end{eqnarray*}

 \noindent Recalling that $\langle F_\alpha, \boldsymbol{r}_\beta \rangle = \langle \For_\alpha, \boldsymbol{e}_\beta \rangle = \mathbbm{1}_{\alpha = \beta}$, (III) along with \textit{(v)} establish \textit{(iii)}:
 $$\text{M}(F,\For) = \text{M}(\textbf{e},\textbf{r})^t = (\mathcal{K}^tJ_\psi)^t = J_\psi \mathcal{K}.$$
 By (II) once again, \textit{(iii)}, and (\ref{hr}), we get \textit{(ii)}:
 $$\text{M}(\For,F) = \text{M}(F,\For)^{-1} = \left(J_\psi \mathcal{K}\right)^{-1} = \mathcal{K}^{-1}J_\psi = \varepsilon\mathcal{K}\epsilon J_\psi.$$

 \noindent Lastly, with (\ref{FM}), (\ref{MF}), and (I) once again, we finally establish \textit{(i)}:
 \begin{eqnarray*}
 \text{M}(\For,M) &=& \text{M}(\For,F)\text{M}(F,M)
 \\ &=& \mathcal{K}^{-1}J_\psi \mathcal{K}
 \\ &=& \text{M}(M,F)\text{M}(F,\For)
 \\ &=& \text{M}(M,\For).
 \end{eqnarray*}
 \end{proof}

\noindent These provide all of the change-of-basis matrices between all the bases considered thus far, excluding both kinds of power sums in each of $\textnormal{QSym}$ and $\textbf{NSym}$.  We now aim to generalize the results on the power sums in $\sym$.

\begin{definition} Let
 \begin{align}
 \mathcal{L}_\phi &= \text{M}(\phi,M); \label{L2}  \\
 \mathcal{L}_\psi &= \text{M}(\psi,M), \label{L1}
 \end{align}
where $\{\phi_\alpha\}$ and $\{\psi_\alpha\}$ are the quasisymmetric power sum bases of the second and first kinds, respectively. 
\end{definition}

\noindent

\noindent The results of the following lemma can be found in \cite{quasipower}, derived from \cite{non}, and can be seen from Theorem \ref{thm:inv_on_phi_psi} and duality.

\begin{lemma} In $\textnormal{QSym}$, for any $\alpha$, both $\psi(\nphi_\alpha) = \varepsilon(\alpha)\phi_\alpha$ and $\omega(\psi_\alpha) = \varepsilon(\alpha)\psi_{\alpha^r}$.
\end{lemma}

\noindent The reversals on the bases appearing in the expansions (\ref{ePsiCOB}) and (\ref{ForPsiCOB}) will force $J_{{\rho}}$ to appear in several of the expressions involving the power sums of the first kind we give in the next theorem, adding some complexity that does not appear in the classical case.

\begin{theorem} In $\textnormal{QSym}$, $$\text{M}(\phi, \For) = \varepsilon \mathcal{L}_\phi \hspace{0.8cm} \text{ and } \hspace{0.4cm} \text{M}(\psi, \For) = \varepsilon \mathcal{L}_\psi J_\rho;$$
$$\text{M}(\phi, F) = \mathcal{L}_\phi \varepsilon\mathcal{K}\varepsilon \hspace{0.5cm} \text{ and } \hspace{0.5cm} \text{M}(\phi, F) = \mathcal{L}_\phi \varepsilon\mathcal{K}\varepsilon.$$

\noindent In $\textbf{NSym}$,
$$\text{M}(\boldsymbol{h},\boldsymbol{\phi}) = z^{-1}\mathcal{L}_\phi^t \hspace{0.5cm} \text{ and } \hspace{0.5cm} \text{M}(\boldsymbol{h}, \boldsymbol{\psi}) = z^{-1}\mathcal{L}_\psi^t;$$
$$\text{M}(\boldsymbol{e},\boldsymbol{\phi}) = \varepsilon z^{-1} \mathcal{L}_\phi^t \hspace{0.9cm} \text{ and } \hspace{0.5cm} \text{M}(\boldsymbol{e},\boldsymbol{\psi}) = \varepsilon z^{-1} J_{{\rho}} \mathcal{L}_\psi^t;$$
$$\text{M}(\boldsymbol{r},\boldsymbol{\phi}) = z^{-1}\mathcal{L}_\phi \varepsilon\mathcal{K}\varepsilon \hspace{0.5cm} \text{ and } \hspace{0.5cm} \text{M}(\boldsymbol{r},\boldsymbol{\psi}) = z^{-1}\mathcal{L}_\psi \varepsilon\mathcal{K}\varepsilon.$$
 \end{theorem}

\begin{proof}
By Lemma \ref{thm:inv_on_phi_psi} and (IV),
$$\text{M}(\phi,\For) = \text{M}(\psi(\phi),\psi(\For)) = \text{M}(\varepsilon\phi,M) = \varepsilon \cdot \text{M}(\phi,M) = \varepsilon \mathcal{L}_\phi.$$
Similarly, but also with use of (I) and (IV),
\begin{eqnarray*}
\text{M}(\psi,\For) &=& \text{M}(\psi,\omega(\For))\text{M}(\omega(\For),\For) 
\\ &=& \text{M}(\omega(\psi),\omega^2(\For))J_\rho 
\\ &=& \text{M}(\varepsilon\psi,\For)J_\rho 
\\ &=& \varepsilon \mathcal{L}_\psi J_\rho.
\end{eqnarray*}
By (I),
$$\text{M}(\phi,F) = \text{M}(\phi,M)\text{M}(M,F) = \mathcal{L}_\phi \mathcal{K}^{-1} = \mathcal{L}_\phi \varepsilon\mathcal{K}\varepsilon;$$
$$\text{M}(\psi,F) = \text{M}(\psi,M)\text{M}(M,F) = \mathcal{L}_\psi \mathcal{K}^{-1} = \mathcal{L}_\psi\varepsilon\mathcal{K}\varepsilon.$$

\noindent Next, by (III),
$$\text{M}(\boldsymbol{h},\boldsymbol{\phi}) = \text{M}(z^{-1}\phi,M)^t = z^{-1} \mathcal{L}_\phi^t;$$
$$\text{M}(\boldsymbol{h},\boldsymbol{\psi}) = \text{M}(z^{-1}\psi,M)^t = z^{-1} \mathcal{L}_\psi^t.$$

\noindent Similarly,
$$\text{M}(\boldsymbol{e},\boldsymbol{\phi}) = \text{M}(z^{-1}\phi,\For)^t = z^{-1}\varepsilon \mathcal{L}_\phi^t;$$
$$\text{M}(\boldsymbol{e},\boldsymbol{\psi}) = \text{M}(z^{-1}\psi,\For)^t = \varepsilon z^{-1} J_{{\rho}} \mathcal{L}_\psi^t.$$

\noindent Lastly, by (III) and (I),
$$\text{M}(\boldsymbol{r},\boldsymbol{\phi}) =  \text{M}(z^{-1}\phi,F)^t = z^{-1}(\text{M}(\phi,M)\text{M}(M,F))^t = z^{-1}(\mathcal{K}^t)^{-1}\mathcal{L}_\phi^t  = z^{-1}\varepsilon\mathcal{K}^t\varepsilon\mathcal{L}_\phi^t;$$
$$\text{M}(\boldsymbol{r},\boldsymbol{\psi}) =  \text{M}(z^{-1}\psi,F)^t = z^{-1}(\text{M}(\psi,M)\text{M}(M,F))^t = z^{-1}(\mathcal{K}^t)^{-1}\mathcal{L}_\psi^t = z^{-1}\varepsilon\mathcal{K}^t\varepsilon\mathcal{L}_\psi^t.$$
\end{proof}

\begin{remark*} We may also give the change-of-basis relations between the two kinds of power sums in both spaces as matrix products, but they would each involve either the inverse of $\mathcal{L}_\phi$ or $\mathcal{L}_\psi$.  For example, in $\textnormal{QSym}$,
$$\text{M}(\phi,\psi) = \text{M}(\phi,M)\text{M}(M,\psi) = \mathcal{L}_\phi \mathcal{L}_\psi^{-1}.$$
\end{remark*}

\noindent We provide a figure similar to Figure \ref{fig:symdiagram} depicting the results from this subsection.  Just as in Figure \ref{fig:symdiagram}, an arrow from an element from the basis $\{b_\alpha\}$ to $\{a_\alpha\}$ is labeled with $\text{M}(a,b)$.  We suppress several edges, including the repetitive ones to/from the power sums of the two kinds.  The hatted terms $\hat{J_\rho}$ along the bottom edges correspond only to the power sum bases of the first kind, $\psi$ and $\boldsymbol{\npsi}$.

\begin{figure}
    $$\begin{tikzpicture}
    \useasboundingbox (-4,-3) rectangle (4,4.5);
      \scope[transform canvas={scale=0.8}]
\node at (-1.85,3) {\Large{$F_\alpha$}};
\node at (1.85,3) {\Large{$\boldsymbol{r}_\alpha$}};
\node at (3.6,1) {\Large{$\boldsymbol{h}_\alpha$}};
\node at (-3.6,1) {\Large{$M_\alpha$}};
\node at (1.85,-3) {\Large{$\boldsymbol{\psi}_\alpha$ $\boldsymbol{\phi}_\alpha$}};
\node at (-1.85,-3) {\Large{$\psi_\alpha$ $ \phi_\alpha$}};
\node at (-3.6,-1) {\Large{$\For_\alpha$}};
\node at (3.6,-1) {\Large{$\boldsymbol{e}_\alpha$}};

\node at (-1.5,5) {\LARGE{QSym}};
\node at (1.5,5) {\LARGE{\textbf{NSym}}};

\draw [thick] (0,4)--(0,-4);

\draw [-latex][<-] plot[smooth, tension=.7] coordinates {(-2.05,-2.7)(-3.4,-1.25)};
\draw [-latex][->] plot[smooth, tension=.7] coordinates {(2.05,-2.7)(3.4,-1.25)};

\draw [-latex][<->] plot[smooth, tension=.7] coordinates {(-3.6,0.7)(-3.6,-0.7)};
\draw [-latex][<->] plot[smooth, tension=.7] coordinates {(3.6,0.7)(3.6,-0.7)};

\draw [-latex][<-] plot[smooth, tension=.7] coordinates {(-2,2.75)(-3.4,1.25)};
\draw [-latex][->] plot[smooth, tension=.7] coordinates {(2,2.75)(3.4,1.25)};

\draw [-latex][->] plot[smooth, tension=.7] coordinates {(-1.85,2.7)(-1.85,-2.7)};
\draw [-latex][<-] plot[smooth, tension=.7] coordinates {(1.85,2.7)(1.85,-2.7)};

\draw [-latex][->] plot[smooth, tension=.7] coordinates {(-3.35,0.75)(-1.95,-2.7)};
\draw [-latex][<-] plot[smooth, tension=.7] coordinates {(3.35,0.75)(1.95,-2.7)};

\node at (-3,1.95) {\footnotesize{$\mathcal{K}$}};
\node at (-2.7,-0.3) {\footnotesize{$\mathcal{L}$}};
\node at (-4.35,0) {\footnotesize{$\mathcal{K}^{-1}J_\psi \mathcal{K}$}};
\node at (-3.25,-2) {\footnotesize{$\varepsilon \mathcal{L} \hat{J_\rho}$}};
\node at (-1.35,0) {\footnotesize{$\mathcal{L}\mathcal{K}^{-1}$}};

\node at (3.1,1.95) {\footnotesize{$\mathcal{K}^t$}};
\node at (2.45,-0.4) {\footnotesize{$z^{-1}\mathcal{L}^t$}};
\node at (4.55,0) {\footnotesize{$\mathcal{K}^tJ_{{\psi}}(\mathcal{K}^t)^{-1}$}};
\node at (3.4,-2.1) {\footnotesize{$\varepsilon z^{-1}  \hat{J_{{\rho}}} \mathcal{L}^t$}};
\node at (1.05,0) {\footnotesize{$z^{-1}(\mathcal{K}^t)^{-1}\mathcal{L}_\psi^t$}};
\endscope
\end{tikzpicture}$$
    \caption{~Matrix Change-of-Basis Expressions in $\textnormal{QSym}$ and $\textbf{NSym}$}
    \label{QNCOB}
\end{figure}
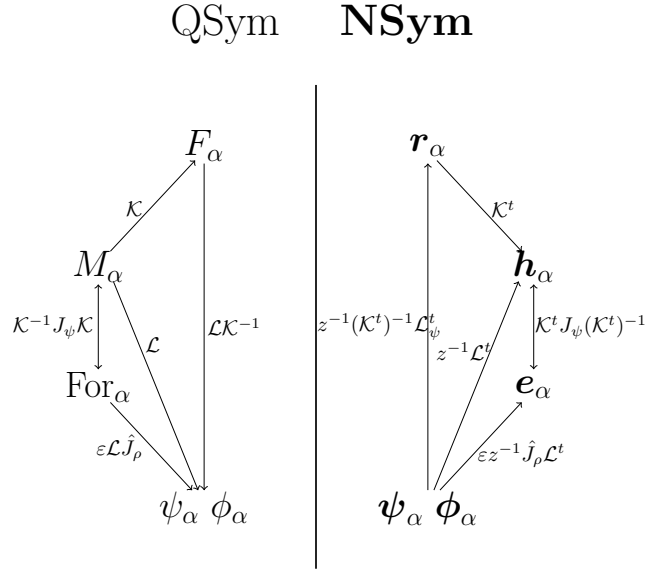

\section{Combinatorial Models for Change of Basis}
\noindent Many of the statistics occurring in the change-of-basis matrices in $\nsym$ are natural generalizations of the statistics on brick tabloids found in \cite{RemTrans,brick}, which give combinatorial descriptions of the transition matrices in $\sym$.  In this section we look at generalizing this work to change of basis in $\qsym$ and $\nsym$.

\subsection{Combinatorial Models in \texorpdfstring{$\sym$}{Sym}}
Before we summarize the results of E\u{g}ecio\u{g}lu and Remmel \cite{brick}, we note that we have taken the liberty of adjusting some of the authors' notation and conventions.  As examples, they choose integer partitions to be written in weakly \textit{increasing} order, and their change-of-basis matrices multiply on the right of row vectors; both of these conventions are opposite of the presentation in this work. 
 Our goal is to generalize the following theorem, which gives a unified combinatorial model for most of the change-of-basis matrices in $\sym$ discussed above.  Let $n$ be a positive integer.
\begin{theorem}[E\u{g}ecio\u{g}lu and Remmel, \cite{brick}]\label{thm:brick}  For $\lambda\vdash n$,
\begin{align} 
    e_\lambda &= \sum_{T \in B_\lambda} (-1)^{|\lambda| - \ell(\operatorname{type}(T))} h_{\operatorname{type(T)}} 
    \vspace{0.2cm}\label{eq:ehbrick}
    \\h_\lambda &= \sum_{T \in B_\lambda} (-1)^{|\lambda| - \ell(\operatorname{type}(T))} e_{\operatorname{type}(T)}
    \\ 
    m_\lambda &= \sum_{T \in B^\lambda} (-1)^{|\operatorname{shape}(T)| - \ell(\lambda)}f_{\operatorname{shape}(T)}
    \vspace{0.2cm}
    \\f_\lambda &= \sum_{T \in B^\lambda} (-1)^{|\operatorname{shape}(T)| - \ell(\lambda)}m_{\operatorname{shape}(T)}
    \\ 
    p_\lambda &= \sum_{T \in B_\lambda} (-1)^{|\lambda| - \ell(\operatorname{type}(T))}w(B_\lambda^{\operatorname{type}(T)})e_{\operatorname{type}(T)}
    \vspace{0.2cm}\label{eq:brickpe}
    \\ 
    p_\lambda &= \sum_{T \in B_\lambda} (-1)^{\ell(\lambda) - \ell(\operatorname{type}(T))}w(B_\lambda^{\operatorname{type}(T)})h_{\operatorname{type}(T)}
    \vspace{0.2cm}
    \\ 
    f_\lambda &= \sum_{T \in B^\lambda} (-1)^{|\lambda| - \ell(\operatorname{type}(T))}\frac{w(B_{\operatorname{shape}(T)}^\lambda)}{z_{\operatorname{shape}(T)}}p_{\operatorname{shape}(T)}
    \vspace{0.2cm}
    \\ 
    m_\lambda &= \sum_{T \in B_\lambda} (-1)^{\ell(\lambda) - \ell(\operatorname{type}(T))}\frac{w(B_\lambda^{\operatorname{type}(T)})}{z_{\operatorname{type}(T)}}p_{\operatorname{type}(T)}
    \vspace{0.2cm}
    \\ 
    p_\lambda &= \sum_{T \in B^\lambda} |OB_{\operatorname{shape}(T)}^\lambda|m_{\operatorname{shape}(T)}
    \vspace{0.2cm}\label{eq:ptombrick}
    \\ 
    p_\lambda &= \sum_{T \in B^\lambda} (-1)^{|\operatorname{type}(T)| - \ell(\lambda)}|OB_{\operatorname{shape}(T)}^\lambda|f_{\operatorname{shape}(T)}
    \vspace{0.2cm}
    \\ 
    h_\lambda &= \sum_{T \in B_\lambda} \frac{|OB_\lambda^{\operatorname{type}(T)}|}{z_{\operatorname{type}(T)}}p_{\operatorname{type(T)}}
    \vspace{0.2cm}
    \\ 
    e_\lambda &= \sum_{T \in B_\lambda} (-1)^{|\lambda| - \ell(\operatorname{type}(T))}\frac{|OB_\lambda^{\operatorname{type}(T)}|}{z_{\operatorname{type}(T)}}p_{\operatorname{type(T)}}
\end{align}
\end{theorem}

\noindent To understand the theorem, the following definitions are necessary.  
\begin{definition}[E\u{g}ecio\u{g}lu and Remmel \cite{brick}]
A \textbf{brick} $b$ of length $k$ is a horizontal strip of $k$ boxes (a Young diagram of shape $(k)$).  If $\mu = (\mu_1,\mu_2,...,\mu_{\ell(\mu)}) \vdash n$, \textbf{associate the set of bricks $\{b_1,b_2,...,b_{\ell(\mu)}\}$ with $\mu$} if brick $|b_i| = \mu_i$ for each $i \in [\ell(\mu)]$.  Then, $T$ is a \textbf{$\mu$-brick tabloid of shape $\lambda$} if $T$ gives a filling of the Young diagram of shape $\lambda \vdash n$ with the set of bricks associated to $\mu$ such that
\\(i) each brick $b_i$ covers exactly $\mu_i$ boxes in a single row of the diagram of shape $\lambda$;
\\(ii) no two bricks overlap.

Say that $B_\lambda$ is the set of all possible brick tabloids of shape $\lambda$, $B^\mu$ is the set of all possible $\mu$-brick tabloids (with \textbf{type} $\mu$), and let $B_\lambda^\mu$ denote the set of all $\mu$-brick tabloids of shape $\lambda$.  
\end{definition}
\begin{note*} In the above definition, it is important to note that bricks of the same size are indistinguishable.
\end{note*}

\begin{example} \label{brtex} Below are the eight $(3,3,2,1)$-brick tabloids of shape $(6,3)$.
$$
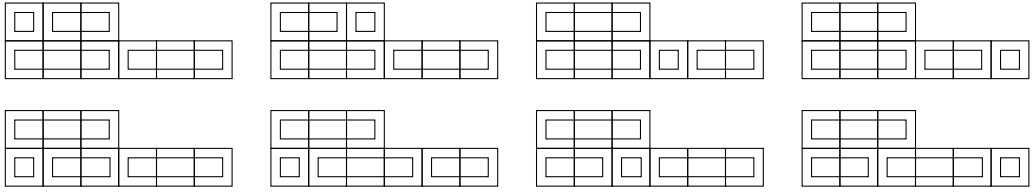
\begin{tikzpicture}
\draw  (0,0) rectangle (.5,.5);
\draw  (.5,0) rectangle (1,.5); 
\draw  (1,0) rectangle (1.5,.5);
\draw  (1.5,0) rectangle (2,.5);
\draw  (2,0) rectangle (2.5,.5);
\draw  (2.5,0) rectangle (3,.5);
\draw  (0,.5) rectangle (.5,1);
\draw  (.5,.5) rectangle (1,1);
\draw  (1,.5) rectangle (1.5, 1);
\draw  (.125, 0.125) rectangle (1.375, .375);
\draw  (1.625, 0.125) rectangle (2.875, .375);
\draw  (.125, .625) rectangle (.375, .875);
\draw  (.625, .625) rectangle (1.375, .875);
\end{tikzpicture}
\hspace{0.5cm}
\begin{tikzpicture}
\draw  (0,0) rectangle (.5,.5);
\draw  (.5,0) rectangle (1,.5); 
\draw  (1,0) rectangle (1.5,.5);
\draw  (1.5,0) rectangle (2,.5);
\draw  (2,0) rectangle (2.5,.5);
\draw  (2.5,0) rectangle (3,.5);
\draw  (0,.5) rectangle (.5,1);
\draw  (.5,.5) rectangle (1,1);
\draw  (1,.5) rectangle (1.5, 1);
\draw  (.125, 0.125) rectangle (1.375, .375);
\draw  (1.625, 0.125) rectangle (2.875, .375);
\draw  (.125, .625) rectangle (.875, .875);
\draw  (1.125, .625) rectangle (1.375, .875);
\end{tikzpicture}
\hspace{0.5cm}
\begin{tikzpicture}
\draw  (0,0) rectangle (.5,.5);
\draw  (.5,0) rectangle (1,.5); 
\draw  (1,0) rectangle (1.5,.5);
\draw  (1.5,0) rectangle (2,.5);
\draw  (2,0) rectangle (2.5,.5);
\draw  (2.5,0) rectangle (3,.5);
\draw  (0,.5) rectangle (.5,1);
\draw  (.5,.5) rectangle (1,1);
\draw  (1,.5) rectangle (1.5, 1);
\draw  (.125, .125) rectangle (1.375, .375);
\draw  (1.625, .125) rectangle (1.875, .375);
\draw  (2.125, .125) rectangle (2.875, .375);
\draw  (.125, .625) rectangle (1.375, .875);
\end{tikzpicture}
\hspace{0.5cm}
\begin{tikzpicture}
\draw  (0,0) rectangle (.5,.5);
\draw  (.5,0) rectangle (1,.5); 
\draw  (1,0) rectangle (1.5,.5);
\draw  (1.5,0) rectangle (2,.5);
\draw  (2,0) rectangle (2.5,.5);
\draw  (2.5,0) rectangle (3,.5);
\draw  (0,.5) rectangle (.5,1);
\draw  (.5,.5) rectangle (1,1);
\draw  (1,.5) rectangle (1.5, 1);
\draw  (.125, .125) rectangle (1.375, .375);
\draw  (1.625, .125) rectangle (2.375, .375);
\draw  (2.625, .125) rectangle (2.875, .375);
\draw  (.125, .625) rectangle (1.375, .875);
\end{tikzpicture}
$$
$$\begin{tikzpicture}
\draw  (0,0) rectangle (.5,.5);
\draw  (.5,0) rectangle (1,.5); 
\draw  (1,0) rectangle (1.5,.5);
\draw  (1.5,0) rectangle (2,.5);
\draw  (2,0) rectangle (2.5,.5);
\draw  (2.5,0) rectangle (3,.5);
\draw  (0,.5) rectangle (.5,1);
\draw  (.5,.5) rectangle (1,1);
\draw  (1,.5) rectangle (1.5, 1);
\draw  (.125, .125) rectangle (0.375, .375);
\draw  (.625, .125) rectangle (1.385, .375);
\draw  (1.625, .125) rectangle (2.875, .375);
\draw  (.125, .625) rectangle (1.375, .875);
\end{tikzpicture}
\hspace{0.5cm}
\begin{tikzpicture}
\draw  (0,0) rectangle (.5,.5);
\draw  (.5,0) rectangle (1,.5); 
\draw  (1,0) rectangle (1.5,.5);
\draw  (1.5,0) rectangle (2,.5);
\draw  (2,0) rectangle (2.5,.5);
\draw  (2.5,0) rectangle (3,.5);
\draw  (0,.5) rectangle (.5,1);
\draw  (.5,.5) rectangle (1,1);
\draw  (1,.5) rectangle (1.5, 1);
\draw  (.125, .125) rectangle (0.375, .375);
\draw  (.625, .125) rectangle (1.875, .375);
\draw  (2.125, .125) rectangle (2.875, .375);
\draw  (.125, .625) rectangle (1.375, .875);
\end{tikzpicture}
\hspace{0.5cm}
\begin{tikzpicture}
\draw  (0,0) rectangle (.5,.5);
\draw  (.5,0) rectangle (1,.5); 
\draw  (1,0) rectangle (1.5,.5);
\draw  (1.5,0) rectangle (2,.5);
\draw  (2,0) rectangle (2.5,.5);
\draw  (2.5,0) rectangle (3,.5);
\draw  (0,.5) rectangle (.5,1);
\draw  (.5,.5) rectangle (1,1);
\draw  (1,.5) rectangle (1.5, 1);
\draw  (.125, .125) rectangle (0.875, .375);
\draw  (1.125, .125) rectangle (1.385, .375);
\draw  (1.625, .125) rectangle (2.875, .375);
\draw  (.125, .625) rectangle (1.375, .875);
\end{tikzpicture} 
\hspace{0.5cm}
\begin{tikzpicture}
\draw  (0,0) rectangle (.5,.5);
\draw  (.5,0) rectangle (1,.5); 
\draw  (1,0) rectangle (1.5,.5);
\draw  (1.5,0) rectangle (2,.5);
\draw  (2,0) rectangle (2.5,.5);
\draw  (2.5,0) rectangle (3,.5);
\draw  (0,.5) rectangle (.5,1);
\draw  (.5,.5) rectangle (1,1);
\draw  (1,.5) rectangle (1.5, 1);
\draw  (.125, .125) rectangle (0.875, .375);
\draw  (1.125, .125) rectangle (2.385, .375);
\draw  (2.625, .125) rectangle (2.875, .375);
\draw  (.125, .625) rectangle (1.375, .875);
\end{tikzpicture}$$
Note the $8$ in parentheses in the equation below, attained from \eqref{eq:ehbrick} with $\lambda = (6,3)$:
 $$\text{M}(e,h)_{(3,3,2,1),(6,3)} = (-1)^{|(6,3)| - \ell((3,3,2,1))}|B_{(6,3)}^{(3,3,2,1)}| = (-1)^{9 - 4}(8) = -8.$$
\end{example}
\begin{definition}
Define a weight function on brick tabloids, $wt: T \rightarrow \mathbb{Z}^+$, as follows.  Let $\mu, \lambda \vdash n$ and let $T$ be any $\mu$-brick tabloid of shape $\lambda$.  If $B(T) = \{b_i\}_{i \in [\ell(\mu)]}$ is the set of bricks associated to $T$, let $B_r(T) \subseteq B(T)$ be the subset of $\ell(\lambda)$ bricks that appear at the rightmost ends of the rows in $T$.  Then, the \textbf{weight} of the brick tabloid $T$ is
$$wt(T) = \prod_{b \in B_r(T)} |b|.$$

\noindent The weight of the entire set of $\mu$-brick tabloids of shape $\lambda$ is the sum of their weights,
$$w(B_\lambda^\mu) = \sum_{T \in B_\lambda^\mu} wt(T).$$
\end{definition}
\begin{example} \label{btwt} The brick tabloids from Example \ref{brtex}, in reading order, have weights $6, 3, 6, 3, 9, 6, 9,$ and $3$, respectively.  Thus $w(B_{(6,3)}^{(3,3,2,1)}) = 45$.  From \eqref{eq:brickpe}, we have
 $$\text{M}(p,e)_{(3,3,2,1),(6,3)} = (-1)^{|(6,3)| - \ell((3,3,2,1))}|B_{(6,3)}^{(3,3,2,1)}| = (-1)^{9 - 4}(45) = -45.$$
\end{example}

 \begin{definition}
Given a partition $\mu = (\mu_1,\mu_2,...,\mu_{\ell(\mu)})$ and associated set of bricks $\{b_i\}_{i\in [\ell(\mu)]}$ ($|b_i| = \mu_i$), index each brick $b_i$ with the subscript $\ell(\mu) - i + 1$, $i \in [\ell(\mu)]$.  That is, index the bricks from smallest to largest with the integers $1,2,...,\ell(\mu)$.
Then, an \textbf{ordered $\mu$-brick tabloid of shape $\lambda$} is a $\mu$-brick tabloid of shape $\lambda$ filled with associated \textit{indexed} bricks such that in each row, the subscripts on the bricks increase from left-to-right. Denote the set of all $\mu$-ordered brick tabloids of shape $\lambda$ by $OB_\lambda^\mu$.
\end{definition}
\begin{example} \label{obtex}
Below are the three $(3,3,2,1)$-ordered brick tabloids of shape $(6,3)$.
$$
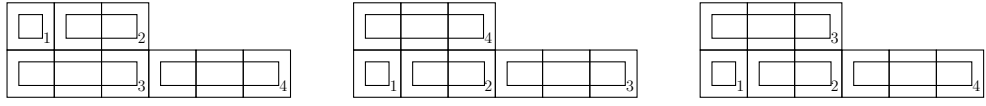
\begin{tikzpicture}[scale=1.25]
\draw  (0,0) rectangle (.5,.5);
\draw  (.5,0) rectangle (1,.5); 
\draw  (1,0) rectangle (1.5,.5);
\draw  (1.5,0) rectangle (2,.5);
\draw  (2,0) rectangle (2.5,.5);
\draw  (2.5,0) rectangle (3,.5);
\draw  (0,.5) rectangle (.5,1);
\draw  (.5,.5) rectangle (1,1);
\draw  (1,.5) rectangle (1.5, 1);
\draw  (.125, 0.125) rectangle (1.375, .375);
\draw  node[scale=0.5] at (1.425,.125) {3};
\draw  (1.625, 0.125) rectangle (2.875, .375);
\draw  node[scale=0.5] at (2.925,.125) {4};
\draw  (.125, .625) rectangle (.375, .875);
\draw  node[scale=0.5] at (.425,.625) {1};
\draw  (.625, .625) rectangle (1.375, .875);
\draw  node[scale=0.5] at (1.425,.625) {2};
\end{tikzpicture}
\hspace{0.8cm}
\begin{tikzpicture}[scale=1.25]
\draw  (0,0) rectangle (.5,.5);
\draw  (.5,0) rectangle (1,.5); 
\draw  (1,0) rectangle (1.5,.5);
\draw  (1.5,0) rectangle (2,.5);
\draw  (2,0) rectangle (2.5,.5);
\draw  (2.5,0) rectangle (3,.5);
\draw  (0,.5) rectangle (.5,1);
\draw  (.5,.5) rectangle (1,1);
\draw  (1,.5) rectangle (1.5, 1);
\draw  (.125, .125) rectangle (0.375, .375);
\draw  node[scale=0.5] at (.425,.125) {1};
\draw  (.625, .125) rectangle (1.385, .375);
\draw  node[scale=0.5] at (1.425,.125) {2};
\draw  (1.625, .125) rectangle (2.875, .375);
\draw  node[scale=0.5] at (2.925,.125) {3};
\draw  (.125, .625) rectangle (1.375, .875);
\draw  node[scale=0.5] at (1.425,.625) {4};
\end{tikzpicture}
\hspace{0.8cm}
\begin{tikzpicture}[scale=1.25]
\draw  (0,0) rectangle (.5,.5);
\draw  (.5,0) rectangle (1,.5); 
\draw  (1,0) rectangle (1.5,.5);
\draw  (1.5,0) rectangle (2,.5);
\draw  (2,0) rectangle (2.5,.5);
\draw  (2.5,0) rectangle (3,.5);
\draw  (0,.5) rectangle (.5,1);
\draw  (.5,.5) rectangle (1,1);
\draw  (1,.5) rectangle (1.5, 1);
\draw  (.125, .125) rectangle (0.375, .375);
\draw  node[scale=0.5] at (.425,.125) {1};
\draw  (.625, .125) rectangle (1.385, .375);
\draw  node[scale=0.5] at (1.425,.125) {2};
\draw  (1.625, .125) rectangle (2.875, .375);
\draw  node[scale=0.5] at (2.925,.125) {4};
\draw  (.125, .625) rectangle (1.375, .875);
\draw  node[scale=0.5] at (1.425,.625) {3};
\end{tikzpicture}$$
Applying \eqref{eq:ptombrick}, we have
$$\text{M}(p,m)_{(6,3),(3,3,2,1)} = |OB_{(6,3)}^{(3,3,2,1)}| = 3.$$
\end{example}

\subsection{Brick Walls and Change of Basis in \texorpdfstring{$\qsym$}{QSym} and \texorpdfstring{$\nsym$}{NSym}.}

This section generalizes the concepts in E\u{g}ecio\u{g}lu and Remmel \cite{brick} to $\qsym$ and $\nsym$, where the statistics from the change-of-basis equations in Gelfand et al.\ \cite{non} are often very natural extensions of the original statistics on brick tabloids. These first four equations easily generalize the first four equations in Theorem \ref{thm:brick}.  Let $n$ be a positive integer.
\begin{theorem} For $\alpha\modelsstrong n$,
\begin{align}
    \boldsymbol{e}_\alpha &= \sum_{W \in \mathcal{W}_\alpha} (-1)^{\ell(\operatorname{type}(W)) - |W|}\boldsymbol{h}_{\operatorname{type}(W)}
\\\boldsymbol{h}_\alpha &= \sum_{W \in \mathcal{W}_\alpha} (-1)^{\ell(\operatorname{type}(W)) - |W|}\boldsymbol{e}_{\operatorname{type}(W)}
\\\For_\alpha &= \sum_{W \in \mathcal{W}^\alpha} (-1)^{\ell(\operatorname{type}(W)) - |W|}M_{\operatorname{sh}(W)}
\\M_\alpha &= \sum_{W \in \mathcal{W}^\alpha} (-1)^{\ell(\operatorname{type}(W)) - |W|}\For_{\operatorname{sh}(W)}
\end{align}
\end{theorem}
\noindent We need the following generalizations of brick tabloids, referred to here as walls, which are simpler in the this case.
\begin{definition}
If $\beta \preceq \alpha$ with $\beta^{(i)} \modelsstrong\alpha_i$, for all $i \in [\ell(\alpha)]$, let the \textit{ordered} set of bricks $B = (b_1,\dots, b_{\ell(\beta)})$ be associated to $\beta$ (where $|b_j| = \beta_j$ for $j \in [\ell(\beta)]$).  Then, the (unique) \textbf{$\beta$-wall of shape $\alpha$}, or \textbf{$\alpha\beta$-wall}, is the filling of the bricks from $B$ into the Young diagram of shape $\alpha$ in order from left-to-right, bottom-up (adopting French notation).
\end{definition}

\noindent It is clear that the $\alpha\beta$-wall exists if and only if $\beta \preceq \alpha$.  (See Definition \ref{def:refinement}.)
For some integers $0 < j_1 < \cdots < j_{\ell(\alpha)} = \ell(\beta)$, the equations below correspond to courses in the $\alpha\beta$-wall.  (A ``course'' is a continuous horizontal stretch of bricks (or stone) laid to build a wall.)
\begin{eqnarray*} \alpha_{\ell(\alpha)} &=& \beta_{j_{\ell(\alpha)-1} + 1} + \beta_{j_{\ell(\alpha)-1} + 2} + \cdots + \beta_{\ell(\beta)}; 
\\ &\vdots&
\\ \alpha_2 &=& \beta_{j_1+1} + \beta_{j_1 + 2} + \cdots + \beta_{j_2};
\\ \alpha_1 &=& \beta_1  + \beta_2 + \cdots + \beta_{j_1}.
\end{eqnarray*}

\begin{example}\label{ex:wall} Let $\alpha = (1,6,2,4)$ and let $\beta = (1,1,3,2,2,3,1)$.  Then $\beta \preceq \alpha$, so the $(1,6,2,4)(1,1,3,2,2,3,1)$-wall exists, and is shown below.
$$
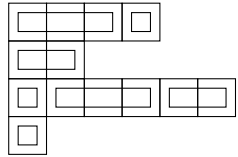
\begin{tikzpicture}
\draw  (0,0) rectangle (.5,.5);
\draw  (0,.5) rectangle (.5,1);
\draw  (.5,.5) rectangle (1,1);
\draw  (1,.5) rectangle (1.5,1);
\draw  (1.5,.5) rectangle (2,1);
\draw  (2,.5) rectangle (2.5,1);
\draw  (2.5,.5) rectangle (3,1);
\draw  (0,1) rectangle (.5,1.5);
\draw  (.5,1) rectangle (1,1.5);
\draw  (0,1.5) rectangle (.5,2);
\draw  (.5,1.5) rectangle (1,2);
\draw  (1,1.5) rectangle (1.5,2);
\draw  (1.5,1.5) rectangle (2,2);
\draw  (.125, .125) rectangle (.375,.375);
\draw  (.125, .625) rectangle (.375,.875);
\draw  (.625, .625) rectangle (1.875,.875);
\draw  (2.125, .625) rectangle (2.875,.875);
\draw  (.125, 1.125) rectangle (.875, 1.375);
\draw  (.125, 1.625) rectangle (1.375, 1.875);
\draw  (1.625, 1.625) rectangle (1.875, 1.875);
\end{tikzpicture}$$
\end{example}

\begin{definition} Say the $\alpha\beta$-wall has \textbf{shape} $\operatorname{sh}(W) = \alpha$, \textbf{size} $|W| = |\alpha|$, and \textbf{type} $\operatorname{type}(W) = \beta$.
 For a fixed composition $\alpha$, let $\mathcal{W}_\alpha$ denote the set of all walls of shape $\alpha$.  Similarly, for composition $\beta$, let $\mathcal{W}^\beta$ be the set of all walls of type $\beta$.  Clearly there are one-to-one correspondences between walls and compositions $\mathcal{W}_\alpha \leftrightarrow \{\beta~|~\beta \preceq \alpha\}$ and $\mathcal{W}^\beta \leftrightarrow \{\alpha~|~\alpha \succeq \beta\}$.\end{definition}
\noindent Although E\u{g}ecio\u{g}lu and Remmel do not mention the Schur functions in their work, the corresponding bases $\{F_\alpha\}_{\alpha\modelsstrong n}$ and $\{\nr_\alpha\}_{\alpha\modelsstrong n}$ fit nicely here into the same framework:
\begin{theorem}For $\alpha\modelsstrong n$,
\begin{align} \boldsymbol{h}_\alpha &= \sum_{W \in \mathcal{W}^\alpha} \boldsymbol{r}_{\operatorname{sh}(W)}
\\ \boldsymbol{r}_\alpha &= \sum_{W \in \mathcal{W}^\alpha} (-1)^{\ell(\operatorname{sh}(W)) - \ell(\operatorname{type}(W))}\boldsymbol{h}_{\operatorname{sh}(W)}
\\ \boldsymbol{e}_\alpha &= \sum_{W \in \mathcal{W}^\alpha} \boldsymbol{r}_{\operatorname{sh}(W)^c} 
\\\boldsymbol{r}_\alpha &= \sum_{W \in \mathcal{W}^{\alpha^c}} (-1)^{\ell(\operatorname{sh}(W))-\ell(\operatorname{type}(W))}\boldsymbol{e}_{\operatorname{sh}(W)}
\\F_\alpha &= \sum_{W \in \mathcal{W}_\alpha} M_{\operatorname{type}(W)}  
\\M_\alpha &= \sum_{W \in \mathcal{W}_\alpha} (-1)^{\ell(\operatorname{type}(W)) - \ell(\operatorname{sh}(W))}F_{\operatorname{type}(W)}
\\F_\alpha &= \sum_{W \in \mathcal{W}_{\alpha^c}} \For_{\operatorname{type}(W)}  
\\\For_\alpha &= \sum_{W \in \mathcal{W}_\alpha} (-1)^{\ell(\operatorname{type}(W))-\ell(\operatorname{sh}(W))}F_{\operatorname{type}(W)^c}
\end{align}
\end{theorem}

\noindent In order to establish combinatorial versions of their other equations, we must define several more statistics on walls.  Three of them, below, are imitations of the weight function defined in \cite{brick} utilizing (\ref{lpr}) (taken from \cite{non}).

\begin{definition}
If $\beta \preceq \alpha$ with $\beta^{(i)} \modelsstrong\alpha_i$ for all $i \in [\ell(\alpha)]$, and $W$ is the $\alpha\beta$-wall, say the \textbf{last parts product} and \textbf{first parts product} of $W$ are
\begin{eqnarray}
\operatorname{lp}(W) &=& \operatorname{lp}(\beta,\alpha) = \prod_{i=1}^{\ell(\alpha)} \beta_{j_i};
\\ \operatorname{fp}(W) &=& \operatorname{lp}(\beta^r,\alpha^r) = \prod_{i=0}^{\ell(\alpha)-1} \beta_{j_i + 1}.
\end{eqnarray}    
\end{definition}

\noindent Thus the statistic $\operatorname{lp}(W)$ (respectively $\operatorname{fp}(W)$) gives the product of the sizes of the bricks at the right (respectively left) ends of the rows in $W$.

A less obvious replacement for the weight function in this context is required to cover the power sums of the second kind.
\begin{definition}If $W$ is an $\alpha, \beta$-wall, let $\pb(W)$ give the product of the number of bricks in each row (or course).
\end{definition}

\begin{example}
    For the wall $W$ in Example \ref{ex:wall}, $\pb(W)=1\cdot 3\cdot 1\cdot 2=6$.
\end{example}
\noindent  
\begin{theorem}For $\alpha\modelsstrong n$,
\begin{eqnarray*}
\boldsymbol{\psi}_\alpha &=& \sum_{W \in \mathcal{W}_\alpha} (-1)^{\ell(\operatorname{type}(W)) - \ell(\operatorname{sh}(W))}\operatorname{lp}(W)\boldsymbol{h}_{\operatorname{type}(W)} 
\\\boldsymbol{\psi}_\alpha &=& \sum_{W \in \mathcal{W}_\alpha} (-1)^{\ell(\operatorname{type}(W)) - \ell(\operatorname{sh}(W))}\operatorname{fp}(W)\boldsymbol{e}_{\operatorname{type}(W)}
\\M_\alpha &=& \sum_{W \in \mathcal{W}^\alpha} (-1)^{\ell(\operatorname{type}(W)) - \ell(\operatorname{sh}(W))}\operatorname{lp}(W) \frac{\psi_{\operatorname{sh}(W)}}{z_{\operatorname{sh}(W)}} 
\\\For_\alpha &=& \sum_{W \in \mathcal{W}^\alpha} (-1)^{\ell(\operatorname{type}(W)) - \ell(\operatorname{sh}(W))} \operatorname{fp}(W)\frac{\psi_{\operatorname{sh}(W)}}{z_{\operatorname{sh}(W)}}
\end{eqnarray*}
\begin{eqnarray*}
\boldsymbol{\phi}_\alpha &=& \sum_{W \in \mathcal{W}_\alpha} (-1)^{\ell(\operatorname{type}(W)) - \ell(\operatorname{sh}(W))} \frac{\prod \operatorname{sh}(W)}{\pb(W)}\boldsymbol{h}_{\operatorname{type}(W)} 
\\\boldsymbol{\phi}_\alpha &=& \sum_{W \in \mathcal{W}_\alpha} (-1)^{|\operatorname{sh}(W)| - \ell(\operatorname{type}(W))} \frac{\prod \operatorname{sh}(W)}{\pb(W)}\boldsymbol{e}_{\operatorname{type}(W)}
\\M_\alpha &=& \sum_{W \in \mathcal{W}^\alpha} (-1)^{\ell(\operatorname{type}(W)) - \ell(\operatorname{sh}(W))} \frac{\prod \operatorname{sh}(W)}{\pb(W)} \frac{\phi_{\operatorname{sh}(W)}}{z_{\operatorname{sh}(W)}}
\\
\For_\alpha &=& \sum_{W \in \mathcal{W}^\alpha} (-1)^{|\operatorname{sh}(W)| - \ell(\operatorname{type}(W))} \frac{\prod \operatorname{sh}(W)}{\pb(W)} \frac{\phi_{\operatorname{sh}(W)}}{z_{\operatorname{sh}(W)}}
\end{eqnarray*}
\end{theorem}
\begin{proof}
    These simply give combinatorial translations of the right-hand equations in the last four lines in both Theorem \ref{thm:nsymCOBeqns} and Corollary \ref{cor:qsymCOBeqns}.
\end{proof}
\noindent Next, we imitate the alteration made on brick tabloids to give ordered brick tabloids in \cite{brick}.  Before we give the definition, we remark that there are always (weakly) fewer $\mu$-ordered brick tabloids of shape $\lambda$ than there are brick tabloids of the same shape and type.  (Consult Examples \ref{brtex} and \ref{obtex}.)  The opposite will be true of our analogous objects, next.

\begin{definition}
Let $\beta \preceq \alpha$ and let $W$ be the $\alpha\beta$-wall.  Then, a \textbf{$\beta$-indexed wall of shape $\alpha$}, or \textbf{indexed $\alpha\beta$-wall}, is an indexing of the bricks in $W$ (associated to $\beta$) in order of increasing size with the integers from $[\ell(\beta)]$.
\end{definition}
\noindent Thus, in an indexed wall, bricks of the same size are distinguishable.

\begin{example} There are four indexed $(2,4,3)(2,2,1,1,3)$-walls, shown below. 
$$
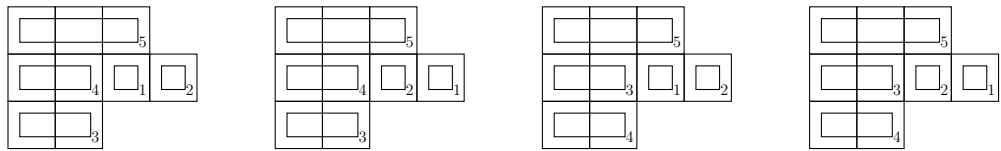
\begin{tikzpicture}[scale=1.25]
\draw  (0,0) rectangle (.5,.5);
\draw  (0.5,0) rectangle (1,.5);
\draw  (0,.5) rectangle (.5,1);
\draw  (.5,.5) rectangle (1,1);
\draw  (1,.5) rectangle (1.5,1);
\draw  (1.5,.5) rectangle (2,1);
\draw  (0,1) rectangle (.5,1.5);
\draw  (.5,1) rectangle (1,1.5);
\draw  (1,1) rectangle (1.5,1.5);
\draw  (.125, .125) rectangle (.875,.375);
\draw  node[scale=0.5] at (.925,.125) {3};
\draw  (.125, .625) rectangle (.875,.875);
\draw  node[scale=0.5] at (.925,.625) {4};
\draw  (1.125, .625) rectangle (1.375,.875);
\draw  node[scale=0.5] at (1.425,.625) {1};
\draw  (1.625, .625) rectangle (1.875, .875);
\draw  node[scale=0.5] at (1.925,.625) {2};
\draw  (.125, 1.125) rectangle (1.375, 1.375);
\draw  node[scale=0.5] at (1.425,1.125) {5};
\end{tikzpicture}
\hspace{1cm}
\begin{tikzpicture}[scale=1.25]
\draw  (0,0) rectangle (.5,.5);
\draw  (0.5,0) rectangle (1,.5);
\draw  (0,.5) rectangle (.5,1);
\draw  (.5,.5) rectangle (1,1);
\draw  (1,.5) rectangle (1.5,1);
\draw  (1.5,.5) rectangle (2,1);
\draw  (0,1) rectangle (.5,1.5);
\draw  (.5,1) rectangle (1,1.5);
\draw  (1,1) rectangle (1.5,1.5);
\draw  (.125, .125) rectangle (.875,.375);
\draw  node[scale=0.5] at (.925,.125) {3};
\draw  (.125, .625) rectangle (.875,.875);
\draw  node[scale=0.5] at (.925,.625) {4};
\draw  (1.125, .625) rectangle (1.375,.875);
\draw  node[scale=0.5] at (1.425,.625) {2};
\draw  (1.625, .625) rectangle (1.875, .875);
\draw  node[scale=0.5] at (1.925,.625) {1};
\draw  (.125, 1.125) rectangle (1.375, 1.375);
\draw  node[scale=0.5] at (1.425,1.125) {5};
\end{tikzpicture}
\hspace{1cm}
\begin{tikzpicture}[scale=1.25]
\draw  (0,0) rectangle (.5,.5);
\draw  (0.5,0) rectangle (1,.5);
\draw  (0,.5) rectangle (.5,1);
\draw  (.5,.5) rectangle (1,1);
\draw  (1,.5) rectangle (1.5,1);
\draw  (1.5,.5) rectangle (2,1);
\draw  (0,1) rectangle (.5,1.5);
\draw  (.5,1) rectangle (1,1.5);
\draw  (1,1) rectangle (1.5,1.5);
\draw  (.125, .125) rectangle (.875,.375);
\draw  node[scale=0.5] at (.925,.125) {4};
\draw  (.125, .625) rectangle (.875,.875);
\draw  node[scale=0.5] at (.925,.625) {3};
\draw  (1.125, .625) rectangle (1.375,.875);
\draw  node[scale=0.5] at (1.425,.625) {1};
\draw  (1.625, .625) rectangle (1.875, .875);
\draw  node[scale=0.5] at (1.925,.625) {2};
\draw  (.125, 1.125) rectangle (1.375, 1.375);
\draw  node[scale=0.5] at (1.425,1.125) {5};
\end{tikzpicture}
\hspace{1cm}
\begin{tikzpicture}[scale=1.25]
\draw  (0,0) rectangle (.5,.5);
\draw  (0.5,0) rectangle (1,.5);
\draw  (0,.5) rectangle (.5,1);
\draw  (.5,.5) rectangle (1,1);
\draw  (1,.5) rectangle (1.5,1);
\draw  (1.5,.5) rectangle (2,1);
\draw  (0,1) rectangle (.5,1.5);
\draw  (.5,1) rectangle (1,1.5);
\draw  (1,1) rectangle (1.5,1.5);
\draw  (.125, .125) rectangle (.875,.375);
\draw  node[scale=0.5] at (.925,.125) {4};
\draw  (.125, .625) rectangle (.875,.875);
\draw  node[scale=0.5] at (.925,.625) {3};
\draw  (1.125, .625) rectangle (1.375,.875);
\draw  node[scale=0.5] at (1.425,.625) {2};
\draw  (1.625, .625) rectangle (1.875, .875);
\draw  node[scale=0.5] at (1.925,.625) {1};
\draw  (.125, 1.125) rectangle (1.375, 1.375);
\draw  node[scale=0.5] at (1.425,1.125) {5};
\end{tikzpicture}
$$
\end{example}

\noindent Note that there is no alteration on the order in which the bricks associated to $\beta$ are laid to build (an indexed) $\alpha\beta$-wall, unlike in the case of ordered brick tabloids. (Compare with Examples \ref{brtex} and \ref{obtex} once more).  

\begin{definition} Let $\mathcal{IW}_\alpha$ denote the set of all ordered walls of shape $\alpha$, let $\mathcal{IW}^\beta$ denote the set of all ordered walls of type $\beta$, and let $\mathcal{IW}_\alpha^\beta$ be the set of indexed $\alpha\beta$-walls.
\end{definition}

\begin{theorem} For any fixed strong compositions $\beta \preceq \alpha$,
$$|\mathcal{IW}_\alpha^\beta| = m_1(\beta)!m_2(\beta)!\cdots m_n(\beta)!.$$
\end{theorem}
\begin{proof} Since the bricks in the $\alpha\beta$-wall must be indexed in order of increasing size, there are $m_i(\beta)!$ ways to index the bricks of size $i$ for each $i \in [n]$.
\end{proof}
\begin{definition}If $W$ is an indexed wall of shape $\alpha$ and type $\beta$, let $\fb(W)$ give the product of the factorial of the number of bricks in each row.
$$\fb(W) = \prod_{i=1}^{\ell(\alpha)}\ell(\beta^{(i)})!$$
\end{definition}

\begin{example}
    For the wall $W$ in Example \ref{ex:wall}, $\fb(W)=1!\cdot 3!\cdot 1!\cdot 2!=12$.
\end{example}
\begin{theorem}For $\alpha\modelsstrong n$,
\begin{eqnarray}
\boldsymbol{h}_\alpha &=& \sum_{W \in \mathcal{IW}_\alpha} \frac{1}{\fb(W)}\frac{\boldsymbol{\phi}_{\operatorname{type}(W)}}{z_{\operatorname{type}(W)}} \nonumber
\\ \boldsymbol{e}_\alpha &=& \sum_{W \in \mathcal{IW}_\alpha} \frac{(-1)^{|\operatorname{sh}(W)| - \ell(\operatorname{type}(W))}}{\fb(W)}\frac{\boldsymbol{\phi}_{\operatorname{type}(W)}}{z_{\operatorname{type}(W)}} \nonumber
\\ \label{CombMPhiCOB}
\phi_\alpha &=& \sum_{W \in \mathcal{IW}^\alpha} \frac{1}{\fb(W)}M_{\operatorname{sh}(W)} \\
\phi_\alpha &=& \sum_{W \in \mathcal{IW}^\alpha} \frac{(-1)^{|\operatorname{sh}(W)| - \ell(\operatorname{type}(W))}}{\fb(W)}\For_{\operatorname{sh}(W)} \nonumber
\end{eqnarray}
\end{theorem}

\noindent The (incredibly) attentive reader will notice that to this point we have not given combinatorial interpretations for a number of change-of-basis equations involving $\{\psi_\alpha\}$.  While it is possible to give such an interpretation, there does not appear to be an analogue of indexed walls that is natural and simplifies their presentation from the original in \cite{non}, so we omit them.

\subsection*{Acknowledgments}
 Both authors wish to thank the anonymous reviewer as well as Darij Grinberg, Steph van Willigenburg, and Nick Mayers for their helpful comments and corrections. The first author wishes to gratefully acknowledge the Institut Mittag-Leffler for graciously hosting her during a portion of this work.

\end{document}